\documentclass[12pt]{amsart}
\usepackage{amsmath}
\usepackage{amsthm}
\usepackage{amssymb}
\usepackage{amscd}
\usepackage{amsfonts}
\usepackage{amsbsy}
\usepackage{epsfig}
\usepackage{graphicx}
\usepackage{tikz}
\usepackage{tikz-cd}
\usetikzlibrary{arrows}
\usepackage{appendix}

\usepackage{pgf}

\usepackage{xcolor}

\newtheorem{theorem}{Theorem}[section]
\newtheorem{corollary}[theorem]{Corollary}
\newtheorem{proposition}[theorem]{Proposition}
\newtheorem{lemma}[theorem]{Lemma}

\theoremstyle{definition}
\newtheorem{remark}[theorem]{Remark}

\newtheorem{definition}[theorem]{Definition}
\newtheorem{example}[theorem]{Example}

\def\ie{{\em i.e.,} }
\def\eg{{\em e.g.} }

\newfont\bbf{msbm10 at 12pt}

\def\phi{\varphi}

\def\N{{\mathbb N}}

\def\theta{\vartheta}

\def\Int{\mbox{\rm Int}}

\begin{document}

\title{Strongly commuting interval maps}

\author{Ana\ Anu\v si\' c and Christopher Mouron}
\address[A.\ Anu\v{s}i\'c]{Departamento de Matem\'atica Aplicada, IME-USP, Rua de Mat\~ao 1010, Cidade Universit\'aria, 05508-090 S\~ao Paulo SP, Brazil}
\email{anaanusic@ime.usp.br}
\address[C.\ Mouron]{Rhodes College, 2000 North Parkway, Memphis, TN 38112, USA}
\email{mouronc@rhodes.edu}
\thanks{AA was supported by grant 2018/17585-5, S\~ao Paulo Research Foundation (FAPESP)}
\date{\today}

\subjclass[2010]{26A21, %Classification  of  real  functions;  Baire  classification of sets and functions
 37E05, %Maps of the interval (piecewise continuous, continuous, smooth)
 54H25, %Fixed-point and coincidence theorems (topological aspects)
 54C10, %Special maps on topological spaces (open, closed,perfect, etc.)
 54C60 %Set-valued  maps
}
\keywords{commuting functions, strongly commuting functions, common fixed point,  set-valued maps, topological entropy, inverse limit space}

\begin{abstract}
	Maps $f,g\colon I\to I$ are called strongly commuting if $f\circ g^{-1}=g^{-1}\circ f$. We show that strongly commuting, piecewise monotone maps $f,g$ can be decomposed into a finite number of invariant intervals (or period 2 intervals) on which $f,g$ are either both open maps, or at least one of them is monotone. As a consequence, we show that strongly commuting piecewise monotone interval maps have a common fixed point. Results of the paper also have implications in understanding dynamical properties of certain maps on inverse limit spaces.
\end{abstract}

\maketitle

\section{Introduction}\label{sec:intro}

We assume that $X$ is a topological space, often a {\em continuum}, i.e. a compact, connected, metric space. A continuous function $f\colon X\to X$ will be called a {\em map}. Two maps $f,g\colon X\to X$ are said to {\em commute} if $f(g(x))=g(f(x))$ for all $x\in X$. The concept of commutativity is very basic, yet there still seems to be a surprising lack of tools for understanding it, even when the space $X$ is very simple, \eg the unit interval $I=[0,1]$. 

The research of commuting interval maps dates back to 1923, when Ritt gave a characterization of commuting polynomials \cite{Ritt}. He showed that commuting polynomials are $\pm x^n, \pm x^n$, for natural numbers $n,m$, Tchebychev polynomials, or iterates of the same polynomial. 
The great advancement in the field happened in the 50s and 60s, related to the question on the existence of a common fixed point of commutating interval maps. 
According to Baxter \cite{Baxter}, the question was raised independently by Dyer in 1954, Shields in 1955, and Dubins in 1956, and in higher generality by Isbell \cite{Isbell}. Precisely, the question was whether for any two commuting maps $f,g\colon I\to I$ there exists $x\in I$ such that $f(x)=g(x)=x$. For example, Ritt's polynomials all have such property.  
The question was settled in the negative independently by Boyce \cite{Boyce} and Huneke \cite{Huneke} in 1967. Detailed overview of the common fixed point problem was given by McDowell in \cite{McDowell}. Modern-day research asks about the existence of other common periodic points of commuting interval maps (see \eg \cite{AliKoo}, \cite{CanovasLinero}), or assumes $X$ is a more complicated space, \eg a square $I^2$ (see \cite{GrincSnoha},\cite{Linero}), or  an arc-like continuum (see \cite{Boronski}, \cite{Mou}). Overview of other contemporary research in the area can be found in \cite{RFBrown}.

In this paper we give a partial answer to the following question (see also Question~2 in \cite{McDowell}):

{\bf Question.}
	What conditions on commuting maps $f,g\colon I\to I$ guarantee the existence of a common fixed point?

For example, Cohen shows in \cite{Cohen} that if $f,g$ are open commuting maps, then they have a common fixed point. More generally, Joichi \cite{Joichi} and Folkman \cite{Folkman} show that it is enough for only one of $f,g$ to be open. 

We introduce a notion of {\em strongly commuting} maps. We say that $f,g\colon X\to X$ are strongly commuting if $f(g^{-1}(x))=g^{-1}(f(x))$ for every $x\in X$. Note that $f$ and $g$ commute if and only if $f(g^{-1}(x))\subseteq g^{-1}(f(x))$ for every $x\in X$, see Lemma~\ref{lem:comm}. So strongly commuting maps are indeed commuting. Our motivation comes from the study of particular maps on the inverse limit spaces, where the strongly commuting maps turned out to have a very important role.

 Let $X$ be a  continuum, and let $f\colon X\to X$ be a map. Then we define the {\em inverse limit space} of $(X,f)$ as:
$$\underleftarrow{\lim}(X,f)=\{(x_0,x_1,x_2,\ldots): f(x_{i+1})=x_i, i\geq 0\}\subseteq X^{\infty}.$$
The space $\underleftarrow{\lim}(X,f)$, when equipped with the product topology, becomes a continuum itself. We are interested in studying dynamical properties of maps on $\underleftarrow{\lim}(X,f)$. For example, there is a natural homeomorphism $\hat f$ on $\underleftarrow{\lim}(X,f)$ given by 
$$\hat f((x_0,x_1,x_2,\ldots))=(f(x_0),f(x_1),f(x_2),\ldots)=(f(x_0),x_0,x_1,\ldots),$$
called {\em shift} or simply {\em natural extension} of $f$. As the name suggests, dynamical properties of the natural extension $\hat f$ resemble those of $f$, \eg the {\em topological entropy} (\ie the exponential growth rate of distinguishable orbits) of $\hat f$ equals the topological entropy of $f$. However, there are other maps (see \eg \cite{Miodus}) on $\underleftarrow{\lim}(X,f)$, beyond the natural extension and its iterates, and there are no available tools for understanding their dynamics.

In \cite{AM-entropy} we study the entropy of the class of maps on $\underleftarrow{\lim}(X,g)$, which we call {\em diagonal maps}. Let $g\colon X\to X$ be a map which commutes with $f\colon X\to X$. Then we define $G\colon\underleftarrow{\lim}(X,f)\to\underleftarrow{\lim}(X,f)$, as:
$$G((x_0,x_1,x_2,\ldots))=(g(x_1),g(x_2),g(x_3),\ldots),$$
see also the commutative diagram in Figure~\ref{fig:diagonal}.

\begin{figure}[!ht]
	\begin{tikzpicture}[->,>=stealth',auto, scale=2]
	\node (1) at (0,1) {$X$};
	\node (2) at (0,0) {$X$};
	\node (3) at (1,1) {$X$};
	\node (4) at (1,0) {$X$};
	\node (5) at (2,1) {$X$};
	\node (6) at (2,0) {$X$};
	\node (7) at (3,1) {$X$};
	\node (8) at (3,0) {$X$};
	\node (9) at (4,1) {$\ldots$};
	\node (10) at (4,0) {$\ldots$};
	\draw [->] (3) to (1);
	\draw [->] (5) to (3);
	\draw [->] (7) to (5);
	\draw [->] (9) to (7);
	\draw [->] (4) to (2);
	\draw [->] (6) to (4);
	\draw [->] (8) to (6);
	\draw [->] (10) to (8);
	\draw [->] (3) to (2);
	\draw [->] (5) to (4);
	\draw [->] (7) to (6);
	\draw [->] (9) to (8);
	
	\node at (0.55,1.15) {\tiny $f$};
	\node at (1.55,1.15) {\tiny $f$};
	\node at (2.55,1.15) {\tiny $f$};
	\node at (3.55,1.15) {\tiny $f$};
	\node at (0.55,-0.2) {\tiny $f$};
	\node at (1.55,-0.2) {\tiny $f$};
	\node at (2.55,-0.2) {\tiny $f$};
	\node at (3.55,-0.2) {\tiny $f$};
	
	%\node at (-0.1,0.5) {\tiny $g_0$};
	\node at (0.65,0.5) {\tiny $g$};
	\node at (1.65,0.5) {\tiny $g$};
	\node at (2.65,0.5) {\tiny $g$};
	\node at (3.65,0.5) {\tiny $g$};
	\end{tikzpicture}
	\caption{Commutative diagram in the construction of a diagonal map.}
	\label{fig:diagonal}
\end{figure}
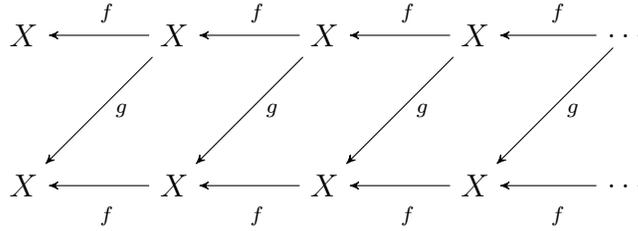 

Since $f$ and $g$ commute, $G$ is well defined and continuous. Moreover, if we assume in addition that $f$ and $g$ strongly commute, then the entropy of $G$ equals the entropy of the set-valued map $g\circ f^{-1}\colon X\to X$, see \cite[Proposition 3.5]{AM-entropy}. This is where it becomes important to understand strongly commuting maps in more depth, preferably well enough to compute the entropy of $g\circ f^{-1}=f^{-1}\circ g$. The results of this paper imply that, when $f$ and $g$ are piecewise monotone and strongly commuting, then the entropy of $g\circ f^{-1}$ exactly equals the maximum of entropies of $f$ and $g$, which is then easy to compute, see \eg \cite{MisSl}. Thus strong commutativity gives tools for understanding a new important class of maps in the inverse limits. For example, diagonal maps were used by Mouron to construct an exact map of the pseudo-arc, \cite{Mou2}. 
 
Section~\ref{sec:tents} gives many examples of strongly commuting maps on the interval. For example, symmetric tent maps $T_3$ and $T_4$ are strongly commuting, but $T_4$ and $T_6$ are not, see Figure~\ref{fig:examples1}. 

\begin{figure}[!ht]
	\centering
	\begin{tikzpicture}[scale=4]
	\draw (0,0)--(0,1)--(1,1)--(1,0)--(0,0);
	\draw (0,0)--(1/3,1)--(2/3,0)--(1,1);
	\node[above] at (1/2,1) {\small $T_3$};
	\end{tikzpicture}
	\begin{tikzpicture}[scale=4]
	\draw (0,0)--(0,1)--(1,1)--(1,0)--(0,0);
	\draw (0,0)--(1/4,1)--(1/2,0)--(3/4,1)--(1,0);
	\node[above] at (1/2,1) {\small $T_4$};
	\end{tikzpicture}
	\begin{tikzpicture}[scale=4]
	\draw (0,0)--(0,1)--(1,1)--(1,0)--(0,0);
	\draw (0,0)--(1/6,1)--(2/6,0)--(3/6,1)--(4/6,0)--(5/6,1)--(1,0);
	\node[above] at (1/2,1) {\small $T_6$};
	\end{tikzpicture}
	
	\begin{tikzpicture}[scale=4]
	\draw (0,0)--(0,1)--(1,1)--(1,0)--(0,0);
	\draw (0,0)--(1,3/4)--(2/3,1)--(0,1/2)--(2/3,0)--(1,1/4)--(0,1);
	\node[above] at (1/2,1) {\small $T_3\circ T_4^{-1}=T_4^{-1}\circ T_3$};
	\end{tikzpicture}
	\begin{tikzpicture}[scale=4]
	\draw (0,0)--(0,1)--(1,1)--(1,0)--(0,0);
	\draw[very thick] (0,0)--(1,2/3)--(1/2,1)--(0,2/3)--(1,0);
	\draw (0,1)--(1,1/3)--(1/2,0)--(0,1/3)--(1,1);
	\node[above] at (1/2,1) {\small $T_6^{-1}\circ T_4\neq T_4\circ T_6^{-1}$};
	\end{tikzpicture}
	\caption{Symmetric tent maps $T_3,T_4,T_6$, and graphs of set-valued functions $T_3\circ T_4^{-1}=T_4^{-1}\circ T_3$, and $T_6^{-1}\circ T_4$. Note that (since $T_4$ and $T_6$ commute) $T_4\circ T_6^{-1}\subseteq T_6^{-1}\circ T_4$. The graph of $T_4\circ T_6^{-1}$ is given in boldface.}
	\label{fig:examples1}
\end{figure}
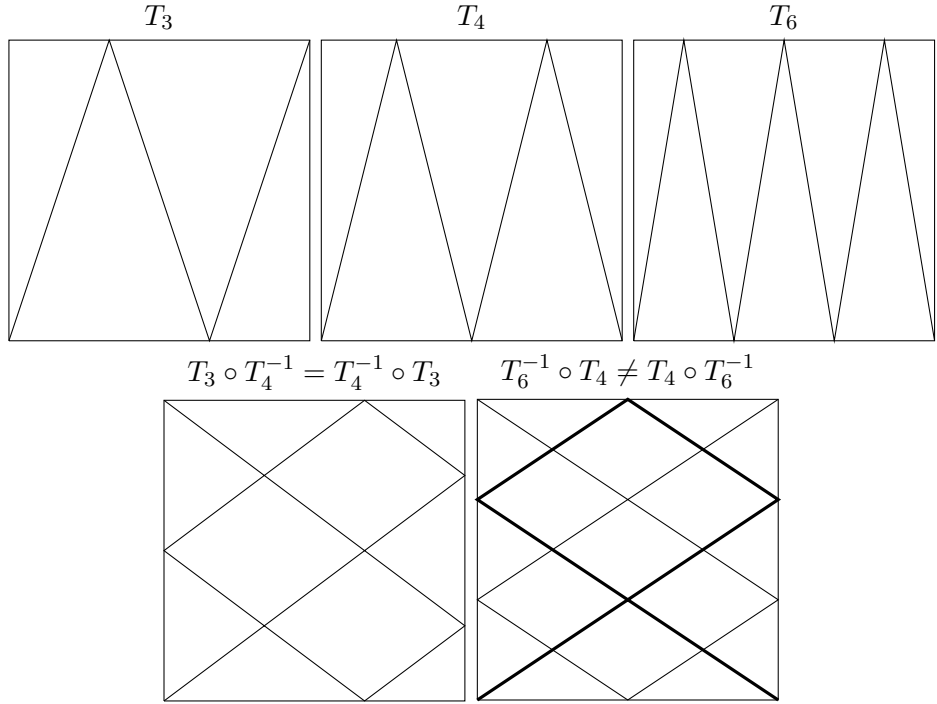

Actually, we show in Corollary~\ref{cor:commoncritpt} that if $f,g$ are piecewise monotone, and strongly commuting, then they don't have a critical point in common. In Figure~\ref{fig:examples2} we show another example of strongly commutating maps - there are invariant intervals for both $f$ and $g$ such that the restrictions are in pairs either both open, or at least one is monotone. The main result of the paper shows that, essentially, this is all that can happen, see also Figures~\ref{fig:typea}, \ref{fig:typeb}, \ref{fig:typec}.

\begin{figure}[!ht]
	\centering
	\begin{tikzpicture}[scale=4]
	\draw (0,0)--(0,1)--(1,1)--(1,0)--(0,0);
	\draw (0,0)--(0,1/3)--(1/3,1/3)--(1/3,0)--(0,0);
	\draw (1/3,1/3)--(1/3,2/3)--(2/3,2/3)--(2/3,1/3)--(1/3,1/3);
	\draw (2/3,2/3)--(2/3,1)--(1,1)--(1,2/3)--(2/3,2/3);
	\draw (0,1/3)--(1/6,0)--(1/3,1/3)--(4/9,5/9)--(5/9,4/9)--(2/3,2/3)--(1,1);
	\node[above] at (1/2,1) {\small $f$};
	\end{tikzpicture}
	\begin{tikzpicture}[scale=4]
	\draw (0,0)--(0,1)--(1,1)--(1,0)--(0,0);
	\draw (0,0)--(0,1/3)--(1/3,1/3)--(1/3,0)--(0,0);
	\draw (1/3,1/3)--(1/3,2/3)--(2/3,2/3)--(2/3,1/3)--(1/3,1/3);
	\draw (2/3,2/3)--(2/3,1)--(1,1)--(1,2/3)--(2/3,2/3);
	\draw (0,0)--(1/9,1/3)--(2/9,0)--(1/3,1/3)--(2/3,2/3)--(5/6,1)--(11/12,5/6)--(1,1);
	\node[above] at (1/2,1) {\small $g$};
	\end{tikzpicture}
	\begin{tikzpicture}[scale=4]
	\draw (0,0)--(0,1)--(1,1)--(1,0)--(0,0);
	\draw (0,0)--(0,1/3)--(1/3,1/3)--(1/3,0)--(0,0);
	\draw (1/3,1/3)--(1/3,2/3)--(2/3,2/3)--(2/3,1/3)--(1/3,1/3);
	\draw (2/3,2/3)--(2/3,1)--(1,1)--(1,2/3)--(2/3,2/3);
	\draw (1/3,0)--(1/9,1/3)--(0,1/6)--(1/9,0)--(1/3,1/3)--(5/9,4/9)--(4/9,5/9)--(2/3,2/3)--(5/6,1)--(11/12,5/6)--(1,1);
	\node[above] at (1/2,1) {\small $f^{-1}\circ g=g\circ f^{-1}$};
	\end{tikzpicture}
	\caption{Maps $f$ and $g$ have invariant intervals $I_1,I_2,I_3$ in common. Note that at least one map in pairs $f|_{I_1}, g|_{I_1}$, $f|_{I_2},g|_{I_2}$ and $f|_{I_3}, g|_{I_3}$ is open - either both maps in a pair are open and non-monotone, or at least one is monotone.}
	\label{fig:examples2}
\end{figure}
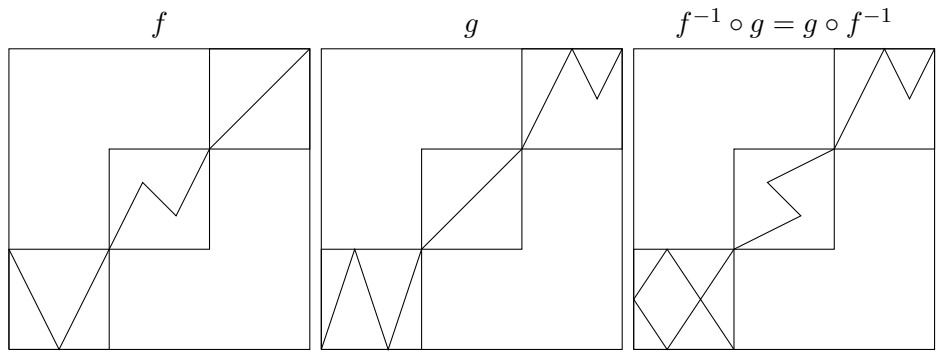

{\bf Theorem~5.21.} {\em Let $f,g\colon I\to I$ be piecewise monotone maps such that $f^{-1}\circ g=g\circ f^{-1}$. Then there are $0=p_0<p_1<\ldots<p_l=1$ such that  one of the following three occurs:
	\begin{enumerate}
		\item[(a)] $[p_i,p_{i+1}]$ is invariant under $f$ and $g$ for every $i$, and
		\begin{itemize}
			\item[(i)] if $g|_{[p_i,p_{i+1}]}$ is open and non-monotone, then $f|_{[p_i,p_{i+1}]}$ is open.
			\item[(ii)] if $g|_{[p_i,p_{i+1}]}$ is non-monotone and not open, then $f|_{[p_i,p_{i+1}]}$ is monotone.
		\end{itemize}
		\item[(b)] $[p_i,p_{i+1}]$ is invariant under  $g$ and $f([p_i,p_{i+1}])=[p_{l-i-1},p_{l-i}]$ for every $i\in\{0,\ldots,l-1\}$, and
		\begin{itemize}
			\item[(i)] if $g|_{[p_i,p_{i+1}]}$ is open and non-monotone, then $f^2|_{[p_i,p_{i+1}]}$ and $f|_{[p_i,p_{i+1}]}$ are open  
			\item[(ii)] if $g|_{[p_i,p_{i+1}]}$ is non-monotone and not open, then $f|_{[p_i,p_{i+1}]}$ and $f|_{[p_{l-i-1},p_{l-i}]}$ are both monotone.
		\end{itemize}
		\item[(c)]  $f([p_i,p_{i+1}])=[p_{l-i-1},p_{l-i}]=g([p_i,p_{i+1}])$ for every $i\in\{0,\ldots,l-1\}$, and
		\begin{itemize}
			\item[(i)] if $g^2|_{[p_i,p_{i+1}]}$ is open and non-monotone, then $f^2|_{[p_i,p_{i+1}]}$ and $f|_{[p_i,p_{i+1}]}$ are open  
			\item[(ii)] if $g^2|_{[p_i,p_{i+1}]}$ is non-monotone and not open, then $f|_{[p_i,p_{i+1}]}$ and $f|_{[p_{l-i-1},p_{l-i}]}$ are both monotone.
		\end{itemize}
		
	\end{enumerate}}

The result of Joichi \cite{Joichi} then easily implies that strongly commuting interval maps have a common fixed point. 

Let us briefly give an outline of the paper. After preliminaries in Section~\ref{sec:preliminaries}, we give many examples in Section~\ref{sec:tents}, and specifically show that symmetric tent maps $T_n$ and $T_m$ are strongly commutating if and only if $n$ and $m$ are relatively prime in Proposition~\ref{prop:mnrelprime} and Proposition~\ref{prop:mnnotrelprime}. In Section~\ref{sec:hats} we introduce the notion of {\em hats} and {\em endpoints} of the graph of the set-valued map $g^{-1}\circ f$ and characterize them in case when $f$ and $g$ are strongly commuting. We note that a term ``hat" seems to be used in older literature for symmetric tent maps, see \eg \cite{Folkman}. Here, a hat is a point in the graph of a set-valued map with certain properties. We believe this tool might have a possible applications in higher generality. Finally, in Section~\ref{sec:main} we prove the main result.

\section{Preliminaries}\label{sec:preliminaries}

Given two topological spaces $X, Y$, a continuous function $f\colon X\to Y$ will be referred to as a {\em map}. We will restrict paper to compact, connected, metrizable spaces (also called {\em continua}), and usually the unit interval $I=[0,1]$.

{\em Set-valued function} from $X$ to $Y$ is a function $F\colon X\to 2^Y$, where $2^Y$ denotes the set of all non-empty closed subsets of $Y$. We always assume that set-valued functions are {\em upper semi-continuous}, \ie the {\em graph} of $F$,
$$\Gamma(F)=\{(x,y): y\in F(x)\}$$
is closed in $X\times Y$. When there is no confusion, we will often abuse the notation and denote set-valued functions as $F\colon X\to Y$. By $\pi_1\colon \Gamma(F)\to X$ and $\pi_2\colon\Gamma(F)\to Y$ we denote the projections to the first and second coordinate.

We say that maps $f,g\colon X\to X$ {\em commute} if $f(g(x))=g(f(x))$ for every $x\in X$.

\begin{lemma}\label{lem:comm}
	Maps $f\colon X\to X$ and $g\colon X\to X$ commute if and only if $f\circ g^{-1}(x)\subseteq g^{-1}\circ f(x)$, for every $x\in X$.
\end{lemma}
\begin{proof}
	Note that $f\circ g^{-1}(x)=\{f(y): g(y)=x\}$ and $g^{-1}\circ f(x)=\{z: g(z)=f(x)\}$. So, given $w=f(y)\in f\circ g^{-1}(x)$, we have
	$g(f(y))=f(g(y))=f(x)$. It follows that $w\in g^{-1}\circ f(x)$.
	
	For the other direction, assume that $f\circ g^{-1}(x)\subseteq g^{-1}\circ f(x)$, for every $x\in X$. Take any $y\in I$ and let $g(y)=z$. Then, since $f(y)\in f\circ g^{-1}(z)\subseteq g^{-1}\circ f(z)$, we have $g\circ f(y)=f(z)=f\circ g(y)$.
\end{proof}

In particular, for commutative maps $f,g\colon X\to X$, it holds that $\Gamma(f\circ g^{-1})\subseteq\Gamma(g^{-1}\circ f)$. 

\begin{remark}\label{rem:graphconnected}
Note that 
$$\Gamma(f\circ g^{-1})=\{(g(t),f(t)): t\in X\},$$
so the graph of $f\circ g^{-1}$ can be parametrized by $X$. In particular, the graph of $f\circ g^{-1}$ is connected. Furthermore, $\Gamma(f\circ g^{-1})$ is closed in $I^2$.
\end{remark}

We say that $f,g\colon X\to X$ are {\em strongly commuting} if $f\circ g^{-1}(x)=g^{-1}\circ f(x)$ for all $x\in X$. Thus strongly commuting maps are commuting, but not vice-versa; see the examples in the following section.

In what follows, we will restrict our functions and set-valued functions to $X=I=[0,1]$, the unit interval. Let $f\colon I\to I$ be a map. We say that $c\in (0,1)$ is a {\em critical point of $f$} if $f|_U$ is not one-to-one for every open interval $U\ni c$. The set of critical points of $f$ will be denoted by $C_f$. We say that $f$ is {\em piecewise monotone} if $C_f$ is finite. We say that $f$ is {\em nowhere constant} if there is no open interval $U\subset I$ for which $f(U)$ is a singleton. Note that every piecewise monotone map is nowhere constant.

In the rest of the paper, we assume that $f,g\colon I\to I$ are piecewise monotone onto maps which strongly commute. Since $f$ and $g$ are nowhere constant, the set $\Gamma(f\circ g^{-1})=\Gamma(g^{-1}\circ f)$ is one-dimensional, and since $f$ and $g$ are continuous, it is closed in $I^2$.
Moreover, we let $n=|C_f|, m=|C_g|$, for $m,n\in\N$, and denote $C_f=\{0<c_1<c_2<\ldots<c_n<1\}, C_g=\{0<d_1<d_2<\ldots<d_m<1\}$. For the sake of notation, we will often denote $c_0=d_0=0, c_{n+1}=d_{m+1}=1$.

\section{Symmetric tent maps}\label{sec:tents}

For $n\geq 2$ we define the {\em symmetric $n$-tent map} $T_n\colon I\to I$ as a map with critical points $\{i/n; 0\leq i\leq n\}$ such that $T_n(i/n)=0$ for even $i$, and $T_n(i/n)=1$ for odd $i$, and such that $T_n|_{[\frac in,\frac{i+1}{n}]}$ is linear for every $i\in\{0,\ldots,n-1\}$. See Figure~\ref{fig:folds}.

\begin{figure}[!ht]
	\centering
	\begin{tikzpicture}[scale=0.4]
	\draw[thick] (0,0)--(0,8)--(8,8)--(8,0)--(0,0);
	\draw[thick] (0,0)--(4,8)--(8,0);
	\node at (4,9) {\small $T_2$};
	\end{tikzpicture}
	\begin{tikzpicture}[scale=0.4]
	\draw[thick] (0,0)--(0,8)--(8,8)--(8,0)--(0,0);
	\draw[thick] (0,0)--(8/3,8)--(16/3,0)--(8,8);
	\node at (4,9) {\small $T_3$};
	\end{tikzpicture}
	\begin{tikzpicture}[scale=0.4]
	\draw[thick] (0,0)--(0,8)--(8,8)--(8,0)--(0,0);
	\draw[thick] (0,0)--(2,8)--(4,0)--(6,8)--(8,0);
	\node at (4,9) {\small $T_4$};
	\end{tikzpicture}
	\begin{tikzpicture}[scale=0.4]
	\draw[thick] (0,0)--(0,8)--(8,8)--(8,0)--(0,0);
	\draw[thick] (0,0)--(8/5,8)--(16/5,0)--(24/5,8)--(32/5,0)--(8,8);
	\node at (4,9) {\small $T_5$};
	\end{tikzpicture}
	\caption{Graphs of symmetric $n$-tent maps $T_n\colon I\to I$.}
	\label{fig:folds}
\end{figure}
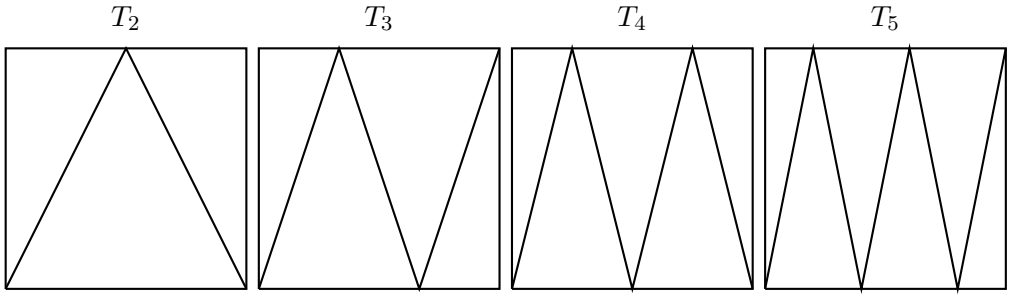

It is not difficult to check that $T_n\circ T_m=T_{nm}$ for every $n,m\geq 2$. It follows that symmetric tent maps commute, \ie $T_n\circ T_m=T_m\circ T_n$, for every $n, m\in\N$. We will show that $T_n$ and $T_m$ strongly commute if and only if $n$ and $m$ are relatively prime. 

Since $T_n$ and $T_m$ commute, Lemma~\ref{lem:comm} implies that $T_n\circ T_m^{-1}(x)\subseteq T_m^{-1}\circ T_n(x)$, for every $x\in I$ and every $n, m\geq 2$. It is, however, in general not true that $T_m^{-1}\circ T_n(x)\subseteq T_n\circ T_m^{-1}(x)$, for every $x\in I$ and $n, m\geq 2$. For example, one easily checks that $T_2\circ T_2^{-1}(x)=x$ and $T_2^{-1}\circ T_2(x)=\{x, 1-x\}$, so $T_2\circ T_2^{-1}(x)\subsetneq T_2^{-1}\circ T_2(x)$ for all $x\in I\setminus\{1/2\}$. See the graphs of some $T_m^{-1}\circ T_n$ in Figure~\ref{fig:examples1}. 

\begin{remark}\label{rem:tent_preimages}
	Note that for $n\geq 2$ and $x,y\in I$ it holds that $T_n(x)=T_n(y)$ if and only if there is $i\geq 0$ such that $x+y=\frac{2i}{n}$, or $|x-y|=\frac{2i}{n}$.
\end{remark}

\begin{remark}\label{rem:more_tent_stuff}
	Let $n,m\in\N$. Since $T_m$ and $T_n$ commute, then $\{0,1\}\supseteq T_m(\{0,1\})\ni T_m(T_n(i/n))=T_n(T_m(i/n))$, for every $i\in\{0,1,\ldots,n\}$. It follows that $T_m(\{i/n: i\in\{0,1,\ldots,n\}\})\subseteq \{i/n: i\in\{0,1,\ldots,n\}\}$.
\end{remark}

\begin{proposition}\label{prop:mnrelprime}
	If $n, m\geq 2$ are relatively prime, then $T_n\circ T_m^{-1}(x)= T_m^{-1}\circ T_n(x)$, for every $x\in I$.
\end{proposition}
\begin{proof}
	Lemma~\ref{lem:comm} implies that $T_n\circ T_m^{-1}(x)\subseteq T_m^{-1}\circ T_n(x)$, for every $x\in I$. Thus we only need to show $T_m^{-1}\circ T_n(x)\subseteq T_n\circ T_m^{-1}(x)$ for every $x\in I$.
	
	Assume first that $x\not\in\{i/n: i\in\{0,1,\ldots,n\}\}$. Then $T_n(x)\in (0,1)$, so $|T_m^{-1}(T_n(x))|=m$. We will show that the cardinality of $T_n\circ T_m^{-1}(x)$ is at least $m$. Since $T_n\circ T_m^{-1}(x)\subseteq T_m^{-1}\circ T_n(x)$, that implies that $T_n\circ T_m^{-1}(x)=T_m^{-1}\circ T_n(x)$.
	Let $a>b\in T_m^{-1}(x)$. We claim that $T_n(a)\neq T_n(b)$. Assume the contrary.
	
	Since $T_m(a)=T_m(b)$, Remark~\ref{rem:tent_preimages} implies that there is $K\geq 0$ such that $a+b=\frac{2K}{m}$, or $a-b=\frac{2K}{m}$. 
	
	{\bf (a)} Assume that $a+b=\frac{2K}{m}$. Since $T_n(a)=T_n(b)$, there is $L\geq 0$ such that $a+b=\frac{2L}{n}$, or $a-b=\frac{2L}{n}$. In the first case we get $\frac{2K}{m}=\frac{2L}{n}$, which is a contradiction with $n,m$ being relatively prime. In the second case we get $a=\frac{K}{m}+\frac{L}{n}$, so by Remark~\ref{rem:more_tent_stuff}, it follows that $x=T_m(a)\in\{T_m(L/n), 1-T_m(L/n)\}\in\{i/n: i\in\{0,1,\ldots,n\}\}$, which is a contradiction.
	
	{\bf (b)} Assume that $a-b=\frac{2K}{m}$. Then again, since $T_n(a)=T_m(b)$, there is $L\geq 0$ such that $a+b=\frac{2L}{n}$, or $a-b=\frac{2L}{n}$. In the second case we get $\frac{2K}{m}=\frac{2L}{n}$, which is a contradiction with $n,m$ being relatively prime. In the first we again get $a=\frac{K}{m}+\frac{L}{n}$, which leads to a contradiction as in case (a).
		
	So we showed that $T_n\circ T_m^{-1}(x)=T_{m}^{-1}\circ T_n(x)$ for every $x\in I\setminus\{i/n: i\in\{0,\ldots,n\}\}=:S$. Furthermore, it is easy to check that $\Gamma(T_m^{-1}\circ T_n)=\overline{\Gamma(T_m^{-1}\circ T_n|_S)}$. Since $\Gamma(T_n\circ T_m^{-1})$ is closed in $I$ (see Remark~\ref{rem:graphconnected}), we have $\Gamma(T_m^{-1}\circ T_n)=\overline{\Gamma(T_m^{-1}\circ T_n|_S)}\subseteq\overline{\Gamma(T_n\circ T_m^{-1}|_S)}\subseteq\overline{\Gamma(T_n\circ T_m^{-1})}=\Gamma(T_n\circ T_m^{-1})$. It follows that $\Gamma(T_m^{-1}\circ T_n)=\Gamma(T_n\circ T_m^{-1})$, and that finishes the proof.
\end{proof}

\begin{proposition}\label{prop:mnnotrelprime}
	If $n,m\geq 2$ are not relatively prime, then there is $x\in I$ such that $T_n\circ T_m^{-1}(x)\subsetneq T_m^{-1}\circ T_n(x)$.
\end{proposition}
\begin{proof}
	Let $x\neq \frac in$, $i\in\{0,\ldots,n\}$. Then $T_n(x)\not\in\{0,1\}$, so $|T_m^{-1}(T_n(x))|=m$. However, $n,m$ are not relatively prime, so there is $k>1$ such that $n=Nk,m=Mk$, for $N,M\in\N$. Then we can write $T_n=T_N\circ T_k, T_m=T_M\circ T_k$, and thus $T_n\circ T_m^{-1}=T_N\circ T_M^{-1}$. However, cardinality of the set $T_N\circ T_M^{-1}(x)$ is at most $M<m$, thus $T_n\circ T_m^{-1}(x)\subsetneq T_m^{-1}\circ T_n(x)$.
\end{proof}

\section{Hats and endpoints}\label{sec:hats}

In this section we study the graphs of set-valued maps $f\circ g^{-1}$ and $g^{-1}\circ f$. We introduce the notions of {\em hats} and {\em endpoints} and then show that, when $f$ and $g$ strongly commute, they exactly correspond to critical points of $f$.

Note that $f^{-1}\circ g=g\circ f^{-1}$ if and only if  $g^{-1}\circ f=f\circ g^{-1}$, so in all of the results $f$ and $g$ can be interchanged and the results are still true. For simplicity, the notions and properties of hats and endpoints will be introduced for the set-valued map $g^{-1}\circ f$.

\begin{definition}
	A point $(x,y)\in\Gamma(g^{-1}\circ f)$ is called a {\em hat of $\Gamma(g^{-1}\circ f)$} if $x\in C_f$ and $y\not\in C_g$. See an example in Figure~\ref{fig:hat}.
\end{definition}

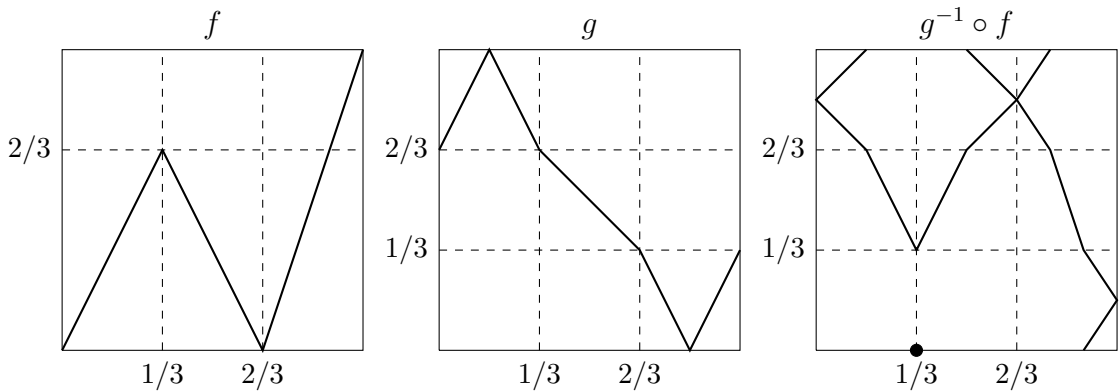
\begin{figure}
	\centering
	\begin{tikzpicture}[scale=4]
	\draw (0,0)--(1,0)--(1,1)--(0,1)--(0,0);
	\draw[thick] (0,0)--(1/3,2/3)--(2/3,0)--(1,1);
	\node[above] at (0.5,1) {$f$};
	
	\draw[dashed] (0,2/3)--(1,2/3);
	\draw[dashed] (1/3,0)--(1/3,1);
	\draw[dashed] (2/3,0)--(2/3,1);
	\node[left] at (0,2/3) {\small $2/3$};
	\node[below] at (1/3,0) {\small $1/3$};
	\node[below] at (2/3,0) {\small $2/3$};
	\end{tikzpicture}
	\begin{tikzpicture}[scale=4]
	\draw (0,0)--(1,0)--(1,1)--(0,1)--(0,0);
	\draw[thick] (0,2/3)--(1/6,1)--(1/3,2/3)--(2/3,1/3)--(5/6,0)--(1,1/3);
	\node[above] at (0.5,1) {$g$};
	
	\draw[dashed] (0,1/3)--(1,1/3);
	\draw[dashed] (0,2/3)--(1,2/3);
	\draw[dashed] (1/3,0)--(1/3,1);
	\draw[dashed] (2/3,0)--(2/3,1);
	\node[left] at (0,1/3) {\small $1/3$};
	\node[left] at (0,2/3) {\small $2/3$};
	\node[below] at (1/3,0) {\small $1/3$};
	\node[below] at (2/3,0) {\small $2/3$};
	\end{tikzpicture}
	\begin{tikzpicture}[scale=4]
	\draw (0,0)--(1,0)--(1,1)--(0,1)--(0,0);
	\draw[thick] (1/6,1)--(0,5/6)--(1/6,2/3)--(1/3,1/3)--(1/2,2/3)--(2/3,5/6)--(7/9,1);
	\draw[thick] (1/2,1)--(2/3,5/6)--(7/9,2/3)--(8/9,1/3)--(1,1/6)--(8/9,0);
	\draw[fill] (1/3,0) circle (0.02);
	\node[above] at (0.5,1) {$g^{-1}\circ f$};
	
	\draw[dashed] (0,1/3)--(1,1/3);
	\draw[dashed] (0,2/3)--(1,2/3);
	\draw[dashed] (1/3,0)--(1/3,1);
	\draw[dashed] (2/3,0)--(2/3,1);
	\node[left] at (0,1/3) {\small $1/3$};
	\node[left] at (0,2/3) {\small $2/3$};
	\node[below] at (1/3,0) {\small $1/3$};
	\node[below] at (2/3,0) {\small $2/3$};
	\end{tikzpicture}
	\caption{Graphs of piecewise linear maps $f$ and $g$ and the corresponding $g^{-1}\circ f$. Note that $(1/3,1/3)\in\Gamma(g^{-1}\circ f)$, and since $1/3\in C_f$ and $1/3\not\in C_g$, it is a hat of $\Gamma(g^{-1}\circ f)$. Note that $f$ and $g$ in this figure do not strongly commute.}
	\label{fig:hat}
\end{figure}

\begin{lemma}\label{lem:hat}
	Let $(x,y)\in\Gamma(g^{-1}\circ f)$ be a hat. Then there is an open disc $\mathcal{O}\ni (x,y)$ in $I^2$ such that $\Gamma(g^{-1}\circ f)\cap\mathcal{O}$ is an open arc $A\subset\Gamma(g^{-1}\circ f)$ for which $y$ is an extremum of $\pi_2(A)$.
\end{lemma} 
\begin{proof}
	Since $(x,y)$ is a hat of $\Gamma(g^{-1}\circ f)$, then $f(x)=g(y)$, and $x\in C_f, y\not\in C_g$. Let $J\subset I$ be an open connected set such that $J\ni f(x)=g(y)$, and denote by $U$ the connected component of $f^{-1}(J)$ which contains $x$, and by $V$ the connected component of $g^{-1}(J)$ which contains $y$. We can assume that $U\cap C_f=\{x\}$, and $V\cap C_g=\emptyset$.  Then obviously $U$ and $V$ are open intervals in $I$, and $g|_V\colon V\to J$ is a homeomorphism.
	
	Let $\mathcal{O}=U\times V$, then $\mathcal{O}$ is an open disc in $I^2$ and $(x,y)\in\mathcal{O}$. Note that $A:=\mathcal{O}\cap\Gamma(g^{-1}\circ f)=\{(t, g^{-1}(f(t))\cap V): t\in U\}$. Furthermore, we claim that $\phi\colon U\to A$ given by $\phi(t)=(t, g^{-1}(f(t))\cap V)$ is a homeomorphism. First take $t\in U$, so $f(t)\in J$. Then recall that $g|_V\colon V\to J$ is a homeomorphism, so there is a unique $s\in V$ such that $g(s)=f(t)$. Thus $\phi$ is well-defined. It is straight-forward to see that $\phi$ is one-to-one and onto. Continuity is also easily checked: if we denote by $\tilde g\colon J\to V$ the inverse of $g|_V$, then we can write $\phi(t)=(t,\tilde g(f(t)))$, and that map is obviously continuous. Thus, $A$ can be parametrized by $U$, so it is an open arc which contains $(x,y)$.
		
	Assume first that $x$ is a local maximum of $f$, so $f(x')<f(x)$ for all $x\neq x'\in U$. If $g|_V$ is increasing, then so is its inverse $\tilde g$, so $\tilde g(f(x'))<\tilde g(f(x))=y$ for all $x'\neq x\in U$. If $g|_V$ is decreasing, then so is its inverse  $\tilde g$, so $\tilde g(f(x'))>\tilde g(f(x))=y$ for all $x'\neq x\in U$. Thus $y$ is an extremum of $\pi_2(A)$. We argue similarly if $x$ is a local minimum of $f$.
\end{proof}

\begin{definition}
	A hat $(x,y)\in\Gamma(g^{-1}\circ f)$ is called an {\em end-hat} if $(x,y)=(g(0),f(0))=(g(1),f(1))$.
\end{definition}

\begin{lemma}\label{lem:param}
	If $g^{-1}\circ f=f\circ g^{-1}$, then for every hat $(x,y)\in\Gamma(g^{-1}\circ f)$ there is either $t=t(x,y)\in C_f$ such that $(x,y)=(g(t),f(t))$, or $(x,y)$ is an end-hat.
\end{lemma}
\begin{proof}
	Fix a hat $(x,y)\in\Gamma(g^{-1}\circ f)$. Recall that $\Gamma(f\circ g^{-1})=\{(a,b): b=f(t), g(t)=a\}=\{(g(t),f(t)): t\in I\}$ and let $T=\{t\in I: (x,y)=(g(t),f(t))\}\neq\emptyset$. 
	
	Assume that there is no $t\in T$ such that $t\in (0,1)$. In that case $T\subset\{0,1\}$. Since $(x,y)$ is contained in an open arc $A$ of $\Gamma(f\circ g^{-1})$ by Lemma~\ref{lem:hat}, and every point of $A$ is of the form $(g(t),f(t))$ for some $t\in I$, we must have  $(x,y)=(g(0),f(0))=(g(1),f(1))$, thus $(x,y)$ is an end-hat.
	
	Assume there is $t\in (0,1)$ such that $t\in T$. By Lemma~\ref{lem:hat}, $f(t)=y$ and $y$ is an extremum of $\pi_2(A)$, so there is a small neighbourhood $U$ of $t$ such that $f(t')<f(t)$ for all $t\neq t'\in U$, or $f(t')>f(t)$ for all $t\neq t'\in U$. Thus $t\in C_f$.
\end{proof}

\begin{remark}\label{rem:hatsnstuff}
	Let $(x,y)\in\Gamma(g^{-1}\circ f)$ be a hat, so it is contained in the interior of an arc $A\subset\Gamma(g^{-1}\circ f)$ by Lemma~\ref{lem:hat}. If $g^{-1}\circ f=f\circ g^{-1}$, then Lemma~\ref{lem:param} implies that we can find $t=t(x,y)\in C_f$ such that $(x,y)=(g(t),f(t))$, or $(x,y)=(g(0),f(0))=(g(1),f(1))$. Thus every hat corresponds to at least one critical point of $f$, or to the set $\{0,1\}$. Obviously that assignment is injective, so the total number of hats in $\Gamma(g^{-1}\circ f)$ is at most $|C_f|+1$.
\end{remark}

\begin{definition}
	A point $(x,y)\in\Gamma(g^{-1}\circ f)$ is called an {\em endpoint of $\Gamma(g^{-1}\circ f)$} if one of the following is satisfied:
	\begin{itemize}
		\item[(a)] $x\in\{0,1\}$, $y\not\in C_g$, or
		\item[(b)] $x\in (0,1)\setminus C_f$, and $y\in\{0,1\}$. 
	\end{itemize}
\end{definition}

\begin{example}
	For maps $f,g$ in Figure~\ref{fig:hat}, the endpoints of $\Gamma(g^{-1}\circ f)$ are exactly $(1/6,1),(1/2,1),(7/9,1),(8/9,0)\in\Gamma(g^{-1}\circ f)$.
\end{example}

\begin{lemma}\label{lem:endpt}
	For every endpoint $(x,y)\in\Gamma(g^{-1}\circ f)$ there is an open disc $\mathcal{O}\ni(x,y)$ in $I^2$ such that $\mathcal{O}\cap\Gamma(g^{-1}\circ f)$ is homeomorphic to a half-open arc $A$ with the endpoint $(x,y)$. Moreover, $x$ is an extremum of all $\pi_1(A)$ and $y$ is an extremum of all $\pi_2(A)$. 
\end{lemma}
\begin{proof}
	We proceed as in the proof of Lemma~\ref{lem:hat}. Note that $f(x)=g(y)$ and that $x\not\in C_f, y\not\in C_f$. Let $J\subset I$ be a connected open set such that $f(x)=g(y)\in J\subset I$. Let $U\subset I$ be a connected component of $f^{-1}(J)$ which contains $x$, and let $V\subset I$ be a connected component of $g^{-1}(J)$ which contains $y$. We can choose $J$ such that $U\cap C_f=V\cap C_g=\emptyset$, so that $f|_U\colon U\to J$ and $g|_V\colon V\to J$ are homeomorphisms. Also $\mathcal{O}:=U\times V$ is an open disc in $I^2$, and $A:=\mathcal{O}\cap\Gamma(g^{-1}\circ f)=\{(t,g^{-1}(f(t))\cap V): t\in U\}=\{(f^{-1}(g(s))\cap U, s): s\in V\}$. As in the proof of Lemma~\ref{lem:hat} we see that $\phi\colon U\to A$ given by $\phi(t)=(t,g^{-1}(f(t))\cap V)$ is a well-defined homeomorphism, and so is $\psi\colon V\to A$ given by $\psi(s)=(f^{-1}(g(s))\cap U,s)$.
	
	Assume that $(x,y)$ is an endpoint of type (a), so $x\in\{0,1\}$. Then, since $U\ni x$ (and assuming that $U\subsetneq I$), we conclude that $A$ is homeomorphic to a half-open interval. Moreover, $\pi_1(A)=U$, so obviously $x$ is an extremum of $\pi_1(A)$. To see that $y$ is an extremum of $\pi_2(A)$, recall that $g|_V\colon V\to J$ is monotone, and thus $g^{-1}(f(x))\cap V=y$ is an extremum of $g^{-1}(f(x'))$ for all $x\neq x'\in U$.
	
	If $(x,y)$ is of type (b), then $y\in\{0,1\}$. Thus, since $V\ni y$, assuming that $V$ contains exactly one of $0,1$, we again conclude that $A$ homeomorphic to a half-open interval. Since $f|_U\colon U\to J$ is monotone, as in the previous paragraph we conclude that $x$ is an extremum of $\pi_1(A)$, and $y$ is an extremum of $\pi_2(A)$.
\end{proof}

\begin{lemma}
	If $f\circ g=g\circ f$, then $\Gamma(g^{-1}\circ f)$ contains at least two endpoints.
\end{lemma}
\begin{proof}
	Assume that $\Gamma(g^{-1}\circ f)$ has zero or one endpoint. Since $g$ is onto, it follows from the Intermediate Value Theorem  that if $g^{-1}(z)\subset C_g\cup\{0,1\}$, then $z\in\{0,1\}$. Assume that $f(0)=z\in(0,1)$. Then $g^{-1}(f(0))\not\subset C_g$, so the set $(0,g^{-1}(f(0)))$ contains an endpoint of $\Gamma(g^{-1}\circ f)$ of type $(a)$. Similarly, if $f(1)\in(0,1)$, the set $(1,g^{-1}(f(1)))$ contains an endpoint of $\Gamma(g^{-1}\circ f)$ of type $(a)$. Since we assumed that there is at most one endpoint, we conclude that at least one of $f(0),f(1)$ is in $\{0,1\}$.
	
	First, assume that exactly one of $f(0),f(1)$ is in $\{0,1\}$. Without loss of generality, assume $f(1)\in (0,1)$. Then there is $y\not\in C_g\cup\{0,1\}$ such that $g(y)=f(1)$, so $(1,y)\in\Gamma(g^{-1}\circ f)$ is an endpoint. 
	
	Assume that $g(0)\in (0,1)$. Then $f^{-1}(g(0))\not\subset C_f$, so there is $z\not\in C_f$ such that $f(z)=g(0)$, so $(z,0)$ is an endpoint of $\Gamma(g^{-1}\circ f)$. Since $y\neq 0$, we conclude that $(1,y)\neq (z,0)$, so we have found at least two endpoints, which is a contradiction. On the other hand, if $g(1)\in(0,1)$, then there is $z\not\in C_f$ such that $f(z)=g(1)$, so $(z,1)$ is an endpoint of $\Gamma(g^{-1}\circ f)$. Since also $y\neq 1$, we conclude that $(1,y)\neq (z,1)$, and conclude that both $g(0),g(1)$ are in $\{0,1\}$. 
	
	If $g(0)=f(0)$ or $g(1)=f(0)$, then $(0,0)$ or $(0,1)$ is a second endpoint, which is a contradiction. Thus, it must be that $g(0)=g(1)\neq f(0)$. However, in that case there are at least two $y_1,y_2\not\in C_g$ such that $g(y_1)=g(y_2)=f(1)$, thus both $(1,y_1),(1,y_2)\in\Gamma(g^{-1}\circ f)$ are endpoints, which is  another contradiction.
	
	Now assume that both $f(0),f(1)$ are in $\{0,1\}$. Then we first conclude that at least one of $g(0),g(1)$ is in $\{0,1\}$, otherwise we find two endpoints. If exactly one of $g(0),g(1)$ is not in $\{0,1\}$, we repeat the previous paragraph with $f$ and $g$ swapped, and again find two endpoints. Thus, both $g(0),g(1)$ have to be in $\{0,1\}$. Note that if $f(i)=g(j)$ for $i,j\in\{0,1\}$, then $(i,j)\in\Gamma(g^{-1}\circ f)$ is an endpoint. Thus if $f(0)\neq f(1)$, then there are at least two endpoints, which is a contradiction. So, assume that $f(0)=f(1)$. Then $g(0)=g(1)\neq f(0)=f(1)$. Let $g(0)=g(1)=i\in\{0,1\}$, and $f(0)=f(1)=1-i$. Then $1-i=f(i)=f(g(0))=g(f(0))=g(i)=i$, which is a contradiction.
\end{proof}

\begin{lemma}\label{lem:endpoints}
	If $g^{-1}\circ f=f\circ g^{-1}$, then
	every endpoint $(x,y)\in\Gamma(g^{-1}\circ f)$ is of the form $(x,y)=(g(t),f(t))$, where $t\in\{0,1\}\cup(C_f\cap C_g)$.
\end{lemma}
\begin{proof}
	Since $\Gamma(f\circ g^{-1})=\{(g(t),f(t)): t\in I\}$, there is $t\in I$ such that $(x,y)=(g(t),f(t))$. If $t\not\in\{0,1\}$, then there is an open $J\subset I$ which contains $t$ and such that $(g(t),f(t))\in U\times V$ for every $t\in J$, where $\mathcal{O}=U\times V$ is a neighbourhood from the proof of Lemma~\ref{lem:endpt}. Thus $f(t)$ is an extremum of $f(J)$ and $g(t)$ is an extremum of $g(J)$, so $t\in C_f\cap C_g$. 
\end{proof}

\begin{remark}\label{rem:nbrs}
Since every endpoint corresponds to a point in $C_f\cup\{0,1\}$, 
we conclude that the sum of number of hats and number of endpoints in $\Gamma(g^{-1}\circ f)=\Gamma(f\circ g^{-1})$ must not exceed $|C_f|+2$. Moreover, note that if there exists an end-hat, then every endpoint $(x,y)$ is of the form $(g(t),f(t))$, where $t\in C_f\cap C_g$. In that case, the number of hats and endpoints must not exceed $|C_f| +1$. 
\end{remark}

For every $c_i\in C_f$, $1\leq i\leq n$, denote the number of hats $(c_i,y)$ by $h_i\geq 0$.
Note that it is the number of non-critical points in $g^{-1}(f(c_i))$. Furthermore, recall that endpoints of $\Gamma(g^{-1}\circ f)$ can only be points $(x,y)$ such that $x\not\in C_f$. Denote by $e_0:=\{(x,y) \text{ endpoint of } \Gamma(g^{-1}\circ f): x\in [0,c_1)\}$, $e_i:=\{(x,y) \text{ endpoint of } \Gamma(g^{-1}\circ f): x\in (c_i,c_{i+1})\}$, for $i<n$, and $e_n:=\{(x,y) \text{ endpoint of } \Gamma(g^{-1}\circ f): x\in (c_n,1]\}$. Then Remark~\ref{rem:nbrs} implies
\begin{equation}\label{eq:1}
\sum_{i=1}^n h_i + \sum_{j=0}^n e_j \leq|C_f|+2= n+2.
\end{equation}

Moreover, if there is an end-hat, then
\begin{equation}\label{eq:2}
\sum_{i=1}^n h_i + \sum_{j=0}^n e_j \leq |C_f|+1= n+1.
\end{equation}

\begin{lemma}\label{lem:bunchofinequalities}
	Let $f,g\colon I\to I$ be piecewise monotone and onto maps. Let $|C_f|=n$, and denote by $\{h_i\}_{i=1}^n$ and $\{e_i\}_{j=0}^n$ the numbers of hats and endpoints of $\Gamma(g^{-1}\circ f)$ as before. Then 
	$$e_0+h_1\geq 2,$$
	$$h_i+e_i+h_{i+1}\geq 2, \text{ for every } 1\leq i<n,$$
	$$h_n+e_n\geq 2.$$	
\end{lemma}
\begin{proof}
	Note that if $f(0)\in (0,1)$, then there is at least one $y\in I$ such that $y\not\in C_g$ and $g(y)=f(0)$, thus $(0,y)\in\Gamma(g^{-1}\circ f)$ is an endpoint. If $f(c_1)\in (0,1)$ we again find $y\in I$ such that $y\not\in C_g$ and $g(y)=f(c_1)$, thus $(c_1,y)$ is a hat. So if $f(0),f(c_1)\in (0,1)$, then $e_0+h_1\geq 2$.
	
	Assume that $f(0)\in\{0,1\}$, and without loss of generality, that $f(0)=0$. If $g(0)\in f([0,c_1])$, then there is $x\in [0,c_1]$ such that $f(x)=g(0)$. Thus, $(x,0)\in\Gamma(g^{-1}\circ f)$. If $x\in[0,c_1)$, then $(x,0)$ is an endpoint, and if $x=c_1$, then $(x,0)$ is a hat. If $g(1)$ is also in $f([0,c_1])$, then we conclude $e_0+h_1\geq 2$. Otherwise there is $y\not\in C_g$ such that $g(y)=f(c_1)$ and $(c_1,y)$ is a hat. Similarly, if $g(0),g(1)$ are neither in $f([0,c_1])$, then there are $y_1,y_2\not\in C_g$ such that $g(y_1)=g(y_2)=f(c_1)$, thus $(c_1,y_1), (c_2,y_2)$ are hats. In any case, $e_0+h_1\geq 2$. Other cases proceed in the same manner.
\end{proof}

When we sum up the equations from the previous lemma we get
\begin{equation}\label{eq:3}
\sum_{j=0}^ne_j+2\sum_{i=1}^nh_i\geq 2(n+1).
\end{equation}
Thus, combined with (\ref{eq:1}), we get
$$2(n+1)\leq\sum_{j=0}^ne_j+2\sum_{i=1}^nh_i\leq 2(n+2)-\sum_{j=0}^ne_j,$$
which implies that $\sum_{j=0}^ne_j\leq 2$.
However, there are at least two endpoints, thus $\sum_{j=0}^ne_j=2$.  Consequently, (\ref{eq:1}) and (\ref{eq:3}) imply that $\sum_{i=1}^nh_i=n$.

Moreover, note that (\ref{eq:3}) and (\ref{eq:2}) imply that the end-hat cannot exist. Otherwise we have 
$$2(n+1)\leq\sum_{j=0}^ne_j+2\sum_{i=1}^nh_i\leq 2(n+1)-\sum_{j=0}^ne_j\leq 2(n+1)-2=2n,$$
which is a contradiction.

\begin{remark}\label{hatpos}
Now we know all the possibilities for $(h_i)_{i=1}^n$. If some $e_i=2$, then $e_j=0$ for $j\neq i$, and $h_{1},h_{2},\ldots,h_{i-1}=2,0,2,0,\ldots,0$, and $h_{i+1},h_{i+2},\ldots, h_n=0,2,0,\ldots,2$. If there are $i<j$ such that $e_i=e_j=1$, then $e_k=0$, for $k\neq i,j$, and either $h_{1},\ldots, h_i=2,0,\ldots,2,0$, $h_{i+1},\ldots,h_j=1,1,\ldots, 1$ $h_{j+1},h_{j+2},\ldots,h_n=0,2,\ldots,0,2$; or $h_k=1$ for all $1\leq k\leq n$.
\end{remark}

\begin{corollary}\label{cor:monotone_param}
	If $f\circ g^{-1}=g^{-1}\circ f$, then $\Gamma(g^{-1}\circ f)$ has exactly two endpoints $(g(0),f(0)),(g(1),f(1))$, and $n=|C_f|$ hats $(g(c_i),f(c_i))$, where $C_f=\{c_1,\ldots,c_n\}$. Moreover, $\alpha\colon I\to\Gamma(g^{-1}\circ f)$, given by $\alpha(t)=(g(t),f(t))$, $t\in I$, is locally monotone, \ie every $t\in I$ has a neighbourhood $t\in U\subset I$ such that $\alpha|_U$ is one-to-one.  
\end{corollary}
\begin{proof}
	First part of the corollary is a direct consequence of Remark~\ref{rem:hatsnstuff} and Lemma~\ref{lem:endpoints}. Assume that $\alpha$ is not locally monotone, so there is an interval $J\subset I$ and $t\in \Int(J)$ such that $\alpha(t)$ is an extremum of $\alpha(J)$. Specifically, this implies that $t\in C_f$, so there is $i\in\{1,\ldots,n\}$ such that $t=c_i$, and hence $\alpha(t)=(g(t),f(t))=(g(c_i),f(c_i))$ is a hat. We take $J=[a_J,b_J]$ small enough such that $\Int(J)\cap C_f=\{c_j\}$, and $\alpha(a_J)=\alpha(b_J)$. 
	Since $\alpha(t)$ is a hat, it has a (closed) neighbourhood in $\Gamma(g^{-1}\circ f)$ which is an arc $A=[a,b]\subset\Gamma(g^{-1}\circ f)$, and we can take $a=\alpha(a_J)$. Such an arc exists by Lemma~\ref{lem:hat}, or simply since $\alpha(t)$ is not an endpoint of $\Gamma(g^{-1}\circ f)$. So we have $\alpha(J)=[a,\alpha(t)]\subsetneq A$. Since every point in $(\alpha(t),b]$ is of the form $(g(t'),f(t'))$, for some $t'\in I$, it follows that there is $t\neq s\in I$ such that $\alpha(t)=\alpha(s)$. Note that $s\in(0,1)$, otherwise $\alpha(t)$ would be an endpoint. Moreover, since $\alpha(t)$ is a hat, Lemma~\ref{lem:hat} implies that there is an open $V\ni s$ such that $f(s)$ is an extremum of $f(V)$, \ie $s\in C_f$. Since $s\neq t$, there is $i\neq j\in\{1,\ldots,n\}$ such that $\alpha(t)=\alpha(s)=(g(c_j),f(c_j))$. But then there are $<n$ hats in $\Gamma(g^{-1}\circ f)$, which is a contradiction.
\end{proof}

\begin{corollary}\label{cor:commoncritpt}
	If $f\circ g^{-1}=g^{-1}\circ f$, then $C_f\cap C_g=\emptyset$.
\end{corollary}
\begin{proof}
	Note that if $s\in C_f\cap C_g$, then $\alpha\colon I\to\Gamma(g^{-1}\circ f)$, given by $\alpha(t)=(g(t),f(t))$, for every $t\in I$, is not locally monotone at $s$, which is in contradiction with Corollary~\ref{cor:monotone_param}.
\end{proof}

\section{The main result}\label{sec:main}

In this section we prove the main result, \ie Theorem~\ref{thm:main}, and consequently the existence of a common fixed point for strongly commuting, piecewise monotone maps in Corollary~\ref{cor:commoncritpt}. The approach is as follows. Assuming that $f,g$ are piecewise monotone, strongly commuting maps, we define {\em primary critical values $\{v_i\}_{i=1}^k$ of $f$} (see Definition~\ref{def:primcritval}) and then count the number of endpoints and hats of $\Gamma(g^{-1}\circ f)$ to show that every $g^{-1}([v_i,v_{i+1}])$ is connected, and moreover $g^{-1}([v_i,v_{i+1}])=[v_i,v_{i+1}]$, see Lemma~\ref{monotonepatr} and Lemma~\ref{invariant}. Actually, we show that either both $f|_{f^{-1}([v_i,v_{i+1}])}$ and $g|_{[v_i,v_{i+1}]}$ are open, or at least one is monotone, see Figure~\ref{fig:primcritval}. Using the primary critical values of $g$, we obtain a similar decomposition of $f$. Then Lemma~\ref{lemma:prim6} gives the main result.

In this section we assume that $f,g\colon I\to I$ are piecewise monotone onto maps such that $f\circ g^{-1}=g^{-1}\circ f$.

\begin{lemma}\label{lem:cross}
	If there are $u< v\in I$ such that $(g(u),f(u))=(g(v),f(v))$, then $g(u)\in C_f$ and $f(u)\in C_g$.
\end{lemma}
\begin{proof}
	Denote $(x,y)=(g(u),f(u))=(g(v),f(v))$. First note that $(x,y)$ is not an endpoint of $\Gamma(g^{-1}\circ f)$, otherwise, as in the proof of Corollary~\ref{cor:monotone_param}, we get that one of $u,v$ is in $C_f$, and consequently, that number of hats is strictly less than $n$.
	
	Assume that $x\not\in C_f$. Then, as in the proof of Lemma \ref{lem:hat}, there is a neighbourhood of $(x,y)$ in $\Gamma(g^{-1}\circ f)$ which is an arc. Recall that $\alpha\colon I\to\Gamma(g^{-1}\circ f)$, given by $\alpha(t)=(g(t),f(t))$, $t\in I$, is locally monotone, thus there are $U=[u_1,u_2]\ni u, V=[v_1,v_2]\ni v$, and a monotone map $\beta\colon U\to V$ such that $\beta(u)=v$, and $g(t)=g(\beta(t)), f(t)=f(\beta(t))$, for all $t\in U$. Take the maximal such $U,V$. Then $U\cap V=\emptyset$, otherwise $\alpha$ is not locally monotone. Assume without loss of generality that $\beta$ is decreasing. Since $U,V$ are maximal, exactly one of $u_2,v_1$ is in $C_g$, or exactly one of $u_2,v_1$ is in $C_f$. Assume that we are in the second case, and $u_2\not\in C_f$, $v_1\in C_f$. Then $(g(v_1),f(v_1))=(g(u_2),f(u_2))$ is a hat of $\Gamma(g^{-1}\circ f)$. However, Lemma~\ref{lem:hat} implies that $u_2\in C_f$, which is a contradiction. We similarly get the contradiction in other cases, thus $x\in C_f$. We similarly prove that $y\in C_g$ also.
\end{proof}

Recall that $C_f=\{0<c_1<\ldots<c_n<1\}$, and we denote $c_0=0, c_{n+1}=1$.

\begin{lemma}\label{lem:connected}
	Sets $g^{-1}(f([c_i,c_{i+1}]))$ are connected for every $0\leq i\leq n+1$.
\end{lemma}
\begin{proof}
	Assume that $g^{-1}(f([c_i,c_{i+1}]))$ has two components, $A,B\subset I$, and assume that $a<b$ for every $a\in A$ and $b\in B$. If one of $A, B$ is degenerate, then there is an isolated point in $\Gamma(g^{-1}\circ f)$, which is a contradiction with $\Gamma(g^{-1}\circ f)=\Gamma(f\circ g^{-1})$ being connected, see Remark~\ref{rem:graphconnected}. Let $A=[a_1,a_2], B=[b_1,b_2]$. Note that $g(a_2),g(b_1)\in \{f(c_i),f(c_{i+1})\}$; $a_1=0$ or $g(a_1)\in \{f(c_i),f(c_{i+1})\}$; and  $b_2=1$ or $g(b_2)\in \{f(c_i),f(c_{i+1})\}$.  Thus $a_1,a_2,b_1,b_2\not\in C_g$. This implies that $h_i+e_i+h_{i+1}\geq 4$, which is a contradiction with Remark~\ref{hatpos}. See Figure~\ref{fig:conn}.
\end{proof}

\begin{figure}
	\centering
	\begin{tikzpicture}
	\draw (-0.1,0.2)--(0,0)--(1,2)--(1.1,1.8);
	\node[below] at (0,0) {\small $c_i$};
	\node[below] at (1,0) {\small $c_{i+1}$};
	\node[above] at (0.7,2) {$f$};
	\node[left] at (-0.2,0) {\small $f(c_i)$};
	\node[left] at (-0.2,2) {\small $f(c_{i+1})$};
	
	\draw[dashed] (-0.2,0)--(6,0);
	\draw[dashed] (-0.2,2)--(6,2);
	
	\draw (2,1)--(2.5,0)--(3.5,2)--(3.6,2.2);
	\draw (3.95,2.15)--(4,2)--(4.5,0.5)--(4.9,1.7)--(5.5,-0.1);
	\node[below] at (3.9,0) {\small $g|_{g^{-1}(f([c_i,c_{i+1}]))}$};
	
	\draw[->] (5.9,1)--(7.5,1);
	\end{tikzpicture}
	\begin{tikzpicture}
	
	\draw[dashed] (0,1.5)--(0,6);
	\draw[dashed] (2,1.5)--(2,6);
	
	\draw (1,2)--(0,2.5)--(2,3.5);
	\draw (2.15,4.05)--(2,4)--(0.5,4.5)--(1.7,4.9)--(0,5.4)--(-0.3,5.3);
	\draw (0,2.5)--(-0.2,2.6);
	\draw (0,2.5)--(-0.2,2.4);
	\draw (2,3.5)--(2.2,3.4);
	\node[below] at (0,1.5) {\small $c_i$};
	\node[below] at (2,1.5) {\small $c_{i+1}$};
	\node[above] at (1,6) {\small $g^{-1}\circ f|_{[c_i,c_{i+1}]}$};
	\end{tikzpicture}
	\caption{An example when $g^{-1}(f([c_i,c_{i+1}])$ is not connected. Here $h_i=1$, $e_i=1$, $h_{i+1}=2$, so $h_i+e_i+h_{i+1}=4$.}
	\label{fig:conn}
\end{figure}
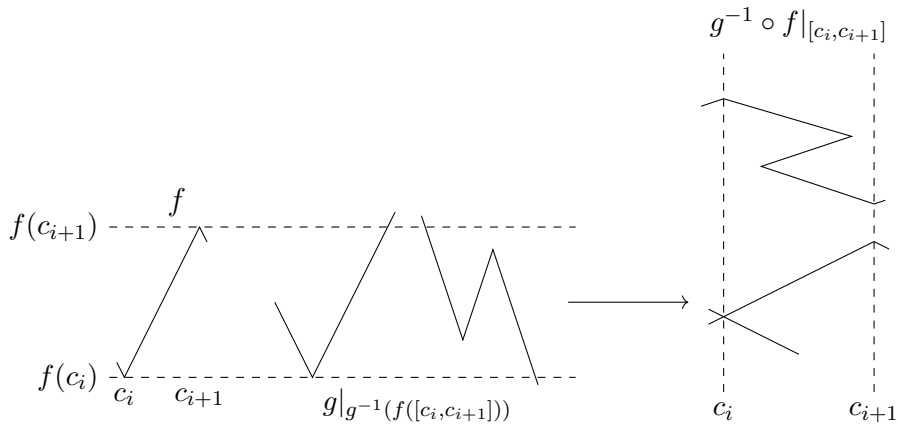

\begin{lemma}\label{lem:g11}
	Assume that $c_i\in C_f$, $0\leq \alpha<c_i$, and $f((\alpha,c_i))=(f(c_i),f(\alpha))$. Furthermore, assume that $$S=\{x\in I: g(x)\in (f(c_i),f(\alpha)), \text{ $x$ local minimum of $g$ or $x\in\{0,1\}$}\}$$ is non-empty. Then $f(x)<\min f(S)$ for all $x\geq c_i$.
\end{lemma}
\begin{proof}
	Denote by $\mu=g(x)=\min f(S)$ and let $y\in (\alpha,c_i)$ be such that $f(y)=\mu$.
	
	Assume first that $x=d_k$ is a local minimum of $g$ for some $k\in\{0,\ldots,m\}$ (so we include $x=0$, but not $x=1$). Note that if there is $z\geq c_i$ such that $g(z)=\mu$, then $y,z\in f^{-1}(g[d_k,d_{k+1}])$. However, $c_i\in (y,z)$ and $f(c_i)<\mu=g(d_k)<g(d_{k+1})$, and thus $f^{-1}(g([d_k,d_{k+1}]))$ is not connected, which is a contradiction with Lemma~\ref{lem:connected}.
	
	On the other hand assume $x=1$. Again if there is $z\geq c_i$ such that $g(z)=\mu$, then $f^{-1}(g([d_m,1]))$ contains $y,z$ but not $c_i\in (y,z)$. Thus it is not connected, and that is a contradiction.
	
	See Figure~\ref{fig:minS}.
\end{proof}

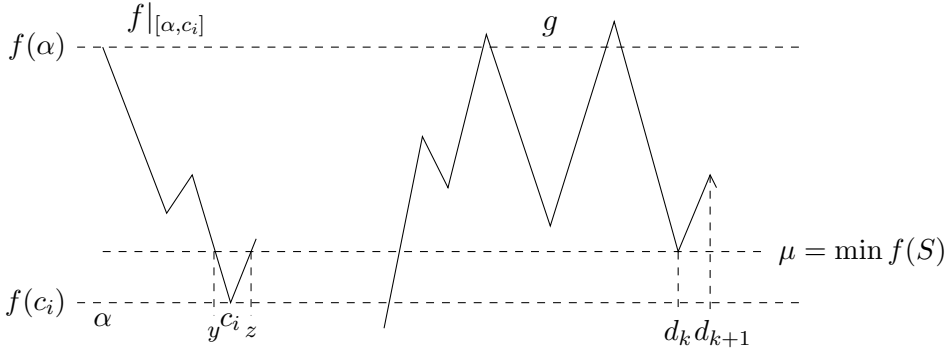
\begin{figure}
	\centering
	\begin{tikzpicture}[scale=1.7]
	\draw (0,2)--(0.5,0.7)--(0.7,1)--(1,0)--(1.2,0.5);
	\draw[dashed] (-0.2,2)--(5.5,2);
	\draw[dashed] (-0.2,0)--(5.5,0);
	\node[below] at (0,0) {\small $\alpha$};
	\node[below] at (1,0) {\small $c_i$};
	\node[left] at (-0.2,0) {\small $f(c_i)$};
	\node[left] at (-0.2,2) {\small $f(\alpha)$};
	\node[above] at (0.5,2) {$f|_{[\alpha,c_i]}$};
	\draw[dashed, very thin] (0.87,0.4)--(0.87,-0.1);
	\node[below] at (0.87,-0.1) {\scriptsize $y$}; 
	\draw[dashed, very thin] (1.16,0.4)--(1.16,-0.1);
	\node[below] at (1.16,-0.1) {\scriptsize $z$}; 
	
	\draw[dashed] (0,0.4)--(5.2,0.4);
	\node[right] at (5.2,0.4) {\small $\mu=\min f(S)$};
	\draw (2.2,-0.2)--(2.5,1.3)--(2.7,0.9)--(3,2.1)--(3.5,0.6)--(4,2.2)--(4.5,0.4)--(4.75,1)--(4.8,0.9);
	\node[above] at (3.5,2) {$g$};
	\node[below] at (4.5,-0.05) {\small $d_k$};
	\draw[dashed, thin] (4.5,0.4)--(4.5,-0.05);
	\node[below] at (4.85,-0.05) {\small $d_{k+1}$};
	\draw[dashed, thin] (4.75,1)--(4.75,-0.05);
	\end{tikzpicture}
	\caption{If there is $z>c_i$ such that $f(z)=\mu$, then $f^{-1}(g([d_k,d_{k+1}]))$ contains $y$ and $z$, but not $c_i$, so it is not connected. See the proof of Lemma~\ref{lem:g11}.}
	\label{fig:minS}
\end{figure}

\begin{proposition}\label{twofold}
	Let $f\colon I\to I$ be a full two-fold, \ie $f(0)=f(1)=0$, and there is $c\in (0,1)$ such that $f(c)=1$, and $f|_{[0,c]}, f|_{[c,1]}$ are monotone. If $g$ is a piecewise monotone, onto map such that $f\circ g^{-1}=g^{-1}\circ f$, then $g$ is an open map.
\end{proposition}
\begin{proof}
	We will first prove that $g(0),g(1)\in\{0,1\}$. Assume that $g(0),g(1)\in (0,1)$. Since $f$ is a full two-fold, there exist distinct  $x_1,x_2\in (0,1)\setminus\{c\}$ and distinct $y_1,y_2\in (0,1)\setminus\{c\}$ be such that $f(x_1)=f(x_2)=g(0)$, and $f(y_1)=f(y_2)=g(1)$. Then $(x_1,0),(x_2,0),(y_1,1),(y_2,1)\in\Gamma(g^{-1}\circ f)$ are different endpoints of $g^{-1}\circ f$. Since $g^{-1}\circ f$ has exactly $2$ endpoints, we get a contradiction.
	
	Thus at least one of $g(0),g(1)$ is in $\{0,1\}$. Assume that $g(0)\in\{0,1\}$ and $g(1)\in (0,1)$. Let $x_1,x_2\in (0,1)\setminus\{c\}$ be such that $f(x_1)=f(x_2)=g(1)$. Then $(x_1,1),(x_2,1)\in\Gamma(g^{-1}\circ f)$ are endpoints of $g^{-1}\circ f$. Thus $(x_1,1)=(g(0),f(0))$, or $(x_2,1)=(g(0),f(0))$, which implies $f(0)=1$. However, we assumed $f(0)=0$, which is a contradiction.
	
	Assume that $g(0)\in(0,1)$, and let again $x_1,x_2\in (0,1)\setminus\{c\}$ be such that $f(x_1)=f(x_2)=g(0)$. Then $(x_1,0),(x_2,0)\in\Gamma(g^{-1}\circ f)$ are endpoints of $g^{-1}\circ f$. Thus $(x_1,0)=(g(1),f(1))$, or $(x_2,0)=(g(1),f(1))$, so $g(1)\in\{x_1,x_2\}$. However, we have proved that $g(1)\in\{0,1\}$, which is a contradiction.
	
	Thus $g(0),g(1)\in\{0,1\}$. Note that if $g(0)=g(1)=0$, then $(0,0),(0,1),(1,0),(1,1)\in\Gamma(g^{-1}\circ f)$ are all endpoints, which contradicts the fact that there are only two endpoints. On the other hand, if $g(0)=g(1)=1$, then there are no endpoints in $\Gamma(g^{-1}\circ f)$. To see that, assume for example that there is an endpoint of the form $(0,y)\in\Gamma(g^{-1}\circ f)$ where $y\not\in C_g$. But if $g(y)=f(0)=0$, then $y$ is a critical point of $g$. Other cases follow the same. Thus we again get a contradiction and conclude that $g(0)\neq g(1)\in\{0,1\}$. In particular, the number of local maxima of $g$ is equal to the number of local minima of $g$, denote it by $N$.
	
	Furthermore, we note that if $d_j\in (0,1)$ is a local maximum of $g$, then $g(d_j)=1$. Otherwise $f^{-1}(g([d_j,d_{j+1}]))$ is not connected. If all local minima of $g$ map to $0$, then $g$ is open and the proof is finished. 
	
	Assume that there is a local minimum $d_j$ of $g$ such that $g(d_j)>0$, see Figure~\ref{fig:twofold}. Thus  $|g^{-1}(0)|<N$. Recall that $N$ denotes the number of local maxima, so $|g^{-1}(1)|=N$. Denote $g^{-1}(1)=\{x_1,\ldots x_N\}$, and note that one of $0,1$ is in $g^{-1}(1)$, and other $x_i$ are critical points of $g$. We have $(g(x_i),f(x_i))\in\Gamma(f\circ g^{-1})=\Gamma(g^{-1}\circ f)$, and they are all the points in $\Gamma(g^{-1}\circ f)$ with first coordinate equal to $1$. Moreover, one of them is going to be an endpoint, and others are hats of $f^{-1}\circ g$, and are thus different. Thus there is exactly $N$ points in $\Gamma(g^{-1}\circ f)$ with first coordinate equal to $1$. Hence, $|g^{-1}(f(1))|=N$. On the other hand, $g^{-1}(f(1))=g^{-1}(0)$, and $|g^{-1}(0)|<N$, which is a contradiction. 
	That finishes the proof.
\end{proof}

\begin{figure}
	\centering
	\begin{tikzpicture}[scale=2.5]
	\draw (0,0)--(0.5,1)--(1,0);
	\node[above] at (0.5,1) {$f$};
	\draw[dashed] (0,0)--(3.2,0);
	\draw[dashed] (0,1)--(3.2,1);
	\node[left] at (0,0) {\small 0};
	\node[left] at (0,1) {\small 1};
	\draw[dashed] (0.5,1)--(0.5,0);
	\node[below] at (0.5,0) {\small $c$};
	
	\draw (1.5,0)--(2,1)--(2.5,0.5)--(3,1);
	\node[above] at (2.25,1) {$g$};
	\draw[dashed] (2,1)--(2,0);
	\node[below] at (2,0) {\small $x_1$};
	\draw[dashed] (2.5,0.5)--(2.5,0);
	\node[below] at (2.5,0) {\small $d_j$};
	\draw[dashed] (3,1)--(3,0);
	\node[below] at (3,0) {\small $x_2$};
	\draw[->] (3.2,0.5)--(3.5,0.5);
	\end{tikzpicture}
	\begin{tikzpicture}[scale=2.5]
	\draw[dashed] (0,-0.1)--(0,1.1);
	\draw[dashed] (0.5,-0.1)--(0.5,1.1);
	\draw[dashed] (1,-0.1)--(1,1.1);
	\draw (0,0)--(0.5,1/3)--(0.25,2/3)--(0.5,1)--(0.75,2/3)--(0.5,1/3)--(1,0);
	\node[below] at (0.5,-0.1) {\small $c$};
	\node[below] at (0,-0.1) {\small $0$};
	\node[below] at (1,-0.1) {\small $1$};
	\node[above] at (0.5,1.1) {\small $g^{-1}\circ f$};
	\end{tikzpicture}
	\caption{If $f\circ g^{-1}=g^{-1}\circ f$ are as in the figure, then $(g(x_1),f(x_1))=(1,f(x_1))$ and $(g(x_2),f(x_2))=(1,f(x_2))$ are two different points of $\Gamma(f\circ g^{-1})=\Gamma(g^{-1}\circ f)$. However, $g^{-1}(f(1))=\{0\}$ is a single point which gives a contradiction. See the proof of Proposition~\ref{twofold}.}
	\label{fig:twofold}
\end{figure}
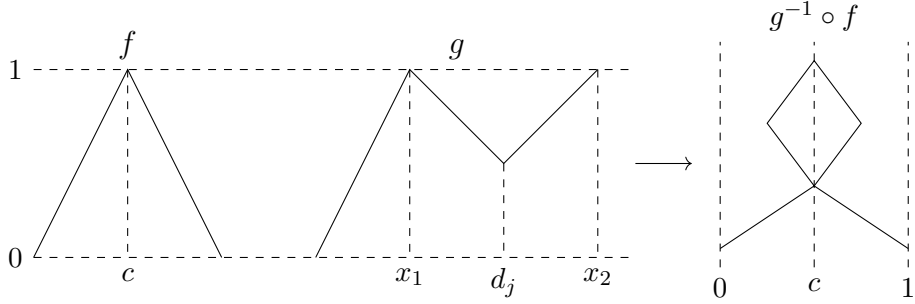

\begin{remark}\label{odd fold} It should be noted that in the previous proof, $g$ is not only open, but also must have an odd number of monotone pieces if it is non-monotone.
\end{remark}

\begin{lemma}\label{prescrit}
	If $d$ is a critical point of $g$, then $f(d)$ is a critical point of $g$.
	Furthermore, if $d\in C_g$ and $y\not\in C_f$,  $(f(d),y)$ is a hat of $f^{-1}\circ g$. (Similarly, if  $c\in C_f$ and $x\not\in C_g$, then $(x,g(c))$ is a hat of $g^{-1}\circ f$.)
\end{lemma}
\begin{proof} It follows from Corollary \ref{cor:commoncritpt} that $d$ is not a critical point of $f$. Hence, there is a open neighborhood $U$ about $d$ such that $f|_U$ is monotone. Suppose on the contrary that $f(d)$ is not a critical point of $g$. Then there exist an open neighborhood $V$ of $f(d)$ such that $g|_V$ is monotone. Thus, by continuity, it follows that there exists an open neighborhood $W$ of $d$ such that $g\circ f|_W$ is monotone. However, $g|_W$ is not monotone and hence $f\circ g|_W= g\circ f|_W$ is not monotone, which is a contradiction. Hence, if $d\in C_g$ and $y\not\in C_f$, then  $(f(d),y)$ is a hat of $f^{-1}\circ g$.
\end{proof}

\begin{lemma}\label{lem:open}
	If $[a,b]\subset I$ is such that $b\in C_f$, and $f|_{[a,b]}$ is a two-fold (\ie $f(a)=f(b)$ and there is a unique critical point $c$ of $f$ in $(a,b)$), then $g|_{g^{-1}(f([a,b]))}$ is an open map.
\end{lemma}
\begin{proof}
	In light of Proposition \ref{twofold} we may assume that $f$ itself is not a full two-fold and hence $[a,b]$ is a proper subinterval of $I$. Also, we may assume that $g$ is not monotone. 
	Let $i\in\{1,\ldots,n\}$ be such that $c=c_{i-1}, b=c_i$.
	By Lemma~\ref{lem:connected}, $g^{-1}(f([c_{i-1},c_{i}]))$ is connected, and denote $g^{-1}(f([c_{i-1},c_i]))=[l,r]$. We assume without loss of generality that $f|_{[a, c_{i-1}]}$ is increasing, \ie $c=c_{i-1}$ is a local maximum. 
	
	We need to exclude the existence of a critical point $d_j\in (l,r)$ of $g$ such that $g(d_j)\neq f(c_{i-1}),f(c_i)$. Note that for every $x\in (l,r)$ it holds that $g(x)\in f([c_{i-1},c_{i}])$. Moreover, if $l=0$, we have to show $g(0)\in\{f(c_{i-1}),f(c_i)\}$, and similarly if $r=1$, we have to show $g(1)\in\{f(c_{i-1}),f(c_i)\}$. Let $d_0=0, d_{m+1}=1$, and assume there is  $d_j\in [l,r]$ such that $g(d_j)\neq f(c_{i-1}),f(c_i)$. If $d_j$ is a local maximum of $g$, then $f^{-1}(g([d_j,d_{j+1}]))$ is not connected, which is a contradiction. Thus every $d_j\in[l,r]$ such that $g(d_j)\not\in\{f(c_{i-1}),f(c_i)\}$ is a local minimum of $g$, or $d_j\in\{0,1\}$.
	
	Actually, note that if $d_j\neq d_0=0$ is as above, then $d_{j-1}\in(l,r)$ and $g(d_{j-1})=f(c_{i-1})$. Similarly, if $l=0$, then $d_1\in(l,r)$ and $g(d_1)=f(c_{i-1})$. Specifically, $g^{-1}(f(c_{i-1}))$ contains at least one critical point of $g$, denote it by $d$.
	
	Let $S$ be as in Lemma~\ref{lem:g11} with $\alpha=c_{i-1}$, and note that $S\neq\emptyset$ by assumption. Let $\mu=\min f(S)=g(d_k)$, for some $k\in\{0,\ldots,m+1\}$. Then Lemma~\ref{lem:g11} implies that $f(x)<\mu$ for all $x\in [0,a)\cup(b,1]$. Specifically, $f(c_{i-1})=1$, and thus $g^{-1}(1)$ contains at least one critical point $d\in C_g$. By Corollary~\ref{cor:commoncritpt}, $d$ is not a critical point of $f$, and thus, by Lemma \ref{prescrit}, $(f(d),g(d))=(f(d),1)\in\Gamma(g\circ f^{-1})=\Gamma(f^{-1}\circ g)$ is a hat of $f^{-1}\circ g$.
	
	First, we are going to show that there is only one point $d'\in [l,r]$, such that $g(d')=f(c_i)$. In fact, we will show that $d'\in \{l,r\}$.
	
	Assume first that there is $d'\in (l,r)$ such that $g(d')=f(c_i)$, see Figure~\ref{fig:op} for example. Since $g^{-1}(f([c_{i-1},c_i]))=[l,r]$,  $d'$ must be a critical point (local minimum) of $g$. Note that $d_j$ is a different critical point of $g$. So $m\geq 2$. Let $j'\in\{1,\ldots,m\}$ be such that $d'=d_{j'}$. Then $g(d_{j'-1})=g(d_{j'+1})=1$, and $g(x)\geq \mu$ for all $x<d_{j'-1}$, and $x>d_{j'+1}$ (otherwise $g^{-1}(f([c_i,c_{i+1}]))$ is not connected). So, since $g$ is onto, it follows that $g(d_{j'})=f(c_i)=0$. Moreover, since $m\geq 2$, at least one of $j'-1,j'+1$ is contained in $\{1,\ldots,m\}$, \ie $d_{j'-1}$ or $d_{j'+1}$ is a critical point of $g$. Specifically, $(f(d_{j'-1}),g(d_{j'-1}))=(f(d_{j'-1}),1)\in\Gamma(f^{-1}\circ g)$, and $(f(d_{j'+1}),g(d_{j'+1}))=(f(d_{j'+1}),1)\in\Gamma(f^{-1}\circ g)$. Thus there are at least two points in $\Gamma(f^{-1}\circ g)$ with second coordinate $1$, and at least one of them is a hat of $f^{-1}\circ g$.
	
	Since $f(1)<\mu$, it follows that $g^{-1}(f(1))$ consists of two non-critical points $x,y$ of $g$, or it is a singleton $d_{j'}$. The first case (see for example Figure~\ref{fig:op}) implies that there are two endpoints $(x,1),(y,1)$ of $f^{-1}\circ g$, and the second gives one hat $(d_{j'},1)$ of $f^{-1}\circ g$. However, we know that $f^{-1}\circ g$ has at least two points with second coordinate $1$, and at least one of them needs to be a hat of $f^{-1}\circ g$, thus we get a contradiction. Thus if $d'\in[l,r]$ is such that $g(d')=f(c_{i})$, then $d'\in\{l,r\}$. 
	
	Next, note that exactly one of $g(l),g(r)$ is equal to $f(c_i)$, otherwise $g^{-1}(f([c_i,c_{i+1}]))$ is not connected. Assume without loss of generality that $g(l)=f(c_i)$, and see Figure~\ref{fig:open2} for an example in this case. Thus $g(x)>f(c_i)$ for all $x>l$. Also, $g(y)<f(c_i)$ for all $y<l$, otherwise, again, $g^{-1}(f([c_i,c_{i+1}]))$ is not connected. 
	
	Let $\tilde{\mu}=\min\{x\in I: g(x)=\mu\}$, and note that $\tilde{\mu}>l$ and $g|_{[l,\tilde{\mu}]}$ is monotone. It follows that if $f(1)\geq f(c_i)$, then $g^{-1}(f(1))$ is a singleton $\{z\}$. That is, $(z,1)\in\Gamma(f^{-1}\circ g)$ is the only point in $\Gamma(f^{-1}\circ g)$ with the second coordinate $1$. Moreover, since $z\not\in C_g$, and $1\not\in C_f$, the point $(z,1)$ is not a hat of $f^{-1}\circ g$. However, we know that there exists at least one hat of $f^{-1}\circ g$ with second coordinate $1$, and that is a contradiction.
	
	We conclude that $f(1)<f(c_i)$ (specifically $f(c_i)>0$).
	Since $c_{i-1}$ is a critical point of $f$ such that $f(c_{i-1})=1$, it follows from Lemma \ref{prescrit}, that $(1,g(c_{i-1}))$ is a hat of $f^{-1}\circ g$. Therefore, $g^{-1}(f(1))$ has to contain at least one critical point of $g$, so there is some $j''\in\{0,1,\ldots,m\}$ such that $d_{j''}<l$. Let $\nu=\max\{g(x): x<l, x\in C_g\cup\{0\}\}$, and note that $f(1)\leq\nu$. It follows that $f(0)>\nu$, otherwise $f^{-1}(g([d_{j'''},d_{j'''+1}]))$ is not connected, where $\nu=g(d_{j'''})$. Specifically, $g^{-1}(f(0))$ consists of a single point $x$ which is not a critical point of $g$, thus $(0,x)$ is an endpoint of $g^{-1}\circ f$. Since $g^{-1}\circ f=f\circ g^{-1}$, point $(0,x)\in g^{-1}\circ f$ equals $(g(t),f(t))$ for some $t\in I$, and since it is an endpoint it must be $t=0$. We conclude that $g(0)=0$ and $g(t)>0$ for all $t>0$.
	
	Now let $\tilde{\nu}=\min\{g(x): x\in C_g\}$. Since $g(C_g)\subset (0,1]$, it follows that $\tilde{\nu}>0$. Now note that $f(1)\neq 0$, otherwise again $g^{-1}(f(1))=\{0\}$ is a singleton and contains no critical point. Since $f$ is onto, there exist a critical point $c_{\iota}\in C_f$, such that $f(c_{\iota})=0$, and take the smallest such $c_{\iota}$. Let $\alpha=c_{i-1}$. Then $\alpha<c_{\iota}$,  $f((\alpha,c_{\iota}))=(0,1)=(f(c_{\iota}), f(\alpha))$, and $$f(c_{\iota})=0<\tilde{\nu}<f(\alpha).$$ Hence, by  Lemma~\ref{lem:g11}, $f(x)<\tilde{\nu}$ for every $x>c_{\iota}$. In particular, $f(1)<\tilde{\nu}$, so $g^{-1}(f(1))$ is a singleton $\{x\}$, where $x$ is not a critical point of $g$, which is again a contradiction. 
\end{proof}

\begin{figure}
	\centering
	\begin{tikzpicture}[scale=1.5]
	\draw (0,0)--(0.5,2)--(1,0)--(1.1,0.35)--(1.2,0)--(1.3,0.31)--(1.4,0)--(1.6,0.3);
	\draw[dashed] (-0.2,2)--(5.5,2);
	\draw[dashed] (-0.2,0)--(5.5,0);
	\node[below] at (0,0) {\small $a$};
	\node[below] at (0.5,0) {\small $c_{i-1}$};
	\node[below] at (1,0) {\small $c_i$};
	\node[left] at (-0.2,0) {\small $f(c_i)$};
	\node[left] at (-0.2,2) {\small $f(c_{i-1})$};
	\node[above] at (0.5,2) {$f$};
	\node[below] at (1.6,0) {\scriptsize 1};
	\node[below] at (1.4,0) {\small $c_n$};
	
	\draw[dashed] (0,0.4)--(5.2,0.4);
	\node[right] at (5.2,0.4) {\small $\mu$};
	\draw (2.2,0.5)--(2.5,1.3)--(2.7,0.9)--(3,2)--(3.5,0)--(4,2)--(4.5,0.4)--(4.75,1)--(4.8,0.9);
	\node[above] at (3.5,2) {$g$};
	\node[below] at (2.2,0) {\small $l$};
	\node[below] at (4.8,0) {\small $r$};
	\node[below] at (3.5,0) {\small $d_{j'}$};
	\node[below] at (3,0) {\small $d_{j'-1}$};
	\node[below] at (4,0) {\small $d_{j'+1}$};
	\draw[dashed] (3,2)--(3,0);
	\draw[dashed] (4,2)--(4,0);
	\end{tikzpicture}
	\begin{tikzpicture}[scale=1.5]
	\draw[dashed] (0,0)--(0,2);
	\node[below] at (0,0) {\small $c_n$};
	\draw[dashed] (0.5,0)--(0.5,2);
	\node[below] at (0.5,0) {\small 1};
	\draw (0.5,1.2)--(0,1)--(0.5,0.8);
	\node[above] at (0.25,2) {\small $g^{-1}\circ f|_{[c_n,1]}$};
	
	\draw[fill] (0.5,1.2) circle(0.03);
	\draw[fill] (0.5,0.8) circle(0.03);
	
	\node[right] at (0.5,1.2) {\small $C$};
	\node[right] at (0.5,0.8) {\small $D$};
	
	\node at (1.1,1) {$\neq$};
	
	\draw[dashed] (1.5,0)--(1.5,2);
	\draw[dashed] (2,0)--(2,2);
	\node[below] at (2,0) {\small 1};
	\node[above] at (1.75,2) {\small $f\circ g^{-1}|_{[c_n,1]}$};
	\draw (1.8,1.6)--(2,1.4)--(1.8,1.2);
	\draw (1.8,0.8)--(2,0.6)--(1.8,0.4);
	\draw[fill] (2,1.4) circle(0.03);
	\draw[fill] (2,0.6) circle(0.03);
	
	\node[right] at (2,1.4) {\small $A$};
	\node[right] at (2,0.6) {\small $B$};
	\end{tikzpicture}
	\caption{Reaching a contradiction in the first part of the proof of Lemma~\ref{lem:open}. Here $A=(g(d_{j'+1}),f(d_{j'+1}))$, $B=(g(d_{j'-1}),f(d_{j'-1}))$, $C=(1,x)$, and $D=(1,y)$.}
	\label{fig:op}
\end{figure}
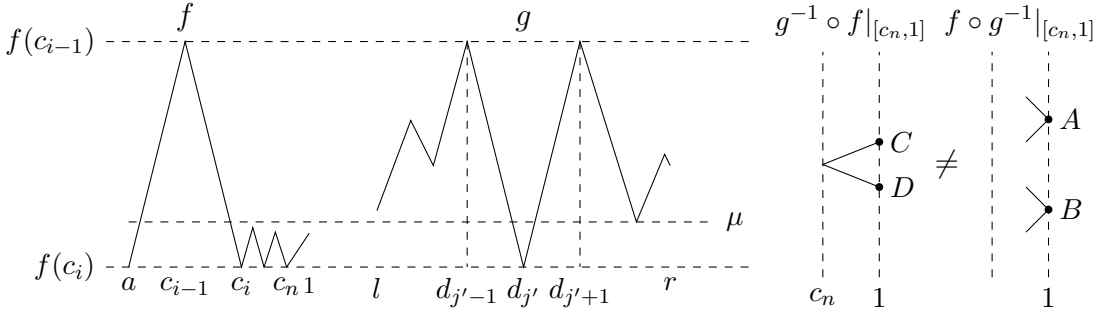

\begin{figure}
	\centering
	\begin{tikzpicture}[scale=1.5]
	\draw (-0.2,-0.2)--(0.5,2)--(1,0)--(1.1,0.35)--(1.4,-2)--(1.5,-1.6)--(1.6,-1.8)--(1.8,-1.5);
	\draw[dashed] (-0.3,2)--(5.5,2);
	\draw[dashed] (-0.3,0)--(5.5,0);
	\draw[dashed] (-0.13,0)--(-0.13,-2);
	\node[below] at (-0.13,-2) {\small $a$};
	\draw[dashed] (0.5,2)--(0.5,-2);
	\node[below] at (0.5,-2) {\small $c_{i-1}$};
	\draw[dashed] (1,0)--(1,-2);
	\node[below] at (1,-2) {\small $c_i$};
	\node[left] at (-0.3,0) {\small $f(c_i)$};
	\node[left] at (-0.3,2) {\small $f(c_{i-1})$};
	\node[above] at (0.5,2) {$f$};
	\draw[dashed] (1.8,-1.5)--(1.8,-2);
	\node[below] at (1.8,-2) {\scriptsize 1};
	\draw[dashed] (1.6,-1.8)--(1.6,-2);
	\node[below] at (1.6,-2) {\small $c_n$};
	\node[below] at (1.4,-2) {\small $c_{\iota}$};
	
	\draw[dashed] (-0.1,0.4)--(5.2,0.4);
	\node[right] at (5.2,0.4) {\small $\mu$};
	\draw[dashed] (-0.1,-1)--(5.2,-1);
	\node[right] at (5.2,-1) {\small $\nu$};
	\draw[dashed] (-0.1,-1.4)--(5.2,-1.4);
	\node[right] at (5.2,-1.4) {\small $\tilde\nu$};
	\draw (2.2,-2)--(2.5,-1)--(2.7,-1.4)--(3.5,2)--(4,0.4);
	\node[above] at (3.5,2) {$g$};
	\draw[dashed] (4,0.4)--(4,-2);
	\node[below] at (4,-2) {\small $r$};
	\draw[dashed] (3.5,2)--(3.5,-2);
	\node[below] at (3.5,-2) {\small $d$};
	\draw[dashed] (3.02,0)--(3.02,-2);
	\node[below] at (3,-2) {\small $l$};
	\node[below] at (2.2,-2) {\scriptsize 0};
	\draw[dashed] (3.5,2)--(3.5,0);
	\end{tikzpicture}
	\begin{tikzpicture}[scale=1.5]
	\draw[dashed] (0,-2)--(0,2);
	\node[below] at (0,-2) {\small $c_n$};
	\draw[dashed] (0.5,-2)--(0.5,2);
	\node[below] at (0.5,-2) {\small 1};
	\draw (0,-1.6)--(0.5,-1.4);
	\node[above] at (0.25,2) {\small $g^{-1}\circ f|_{[c_n,1]}$};
	
	\node at (1,0) {$\neq$};
	
	\draw[dashed] (1.5,-2)--(1.5,2);
	\draw[dashed] (2,-2)--(2,2);
	\node[below] at (2,-2) {\small 1};
	\node[above] at (1.75,2) {\small $f\circ g^{-1}|_{[c_n,1]}$};
	\draw (1.8,-0.3)--(2,-0.5)--(1.8,-0.7);
	\draw[fill] (2,-0.5) circle(0.03);
	
	\node[right] at (2,-0.5) {\small $A$};
	\end{tikzpicture}
	\caption{Reaching a contradiction in the second part of the proof of Lemma~\ref{lem:open}. Here $A=(g(d),f(d))$.}
	\label{fig:open2}
\end{figure}
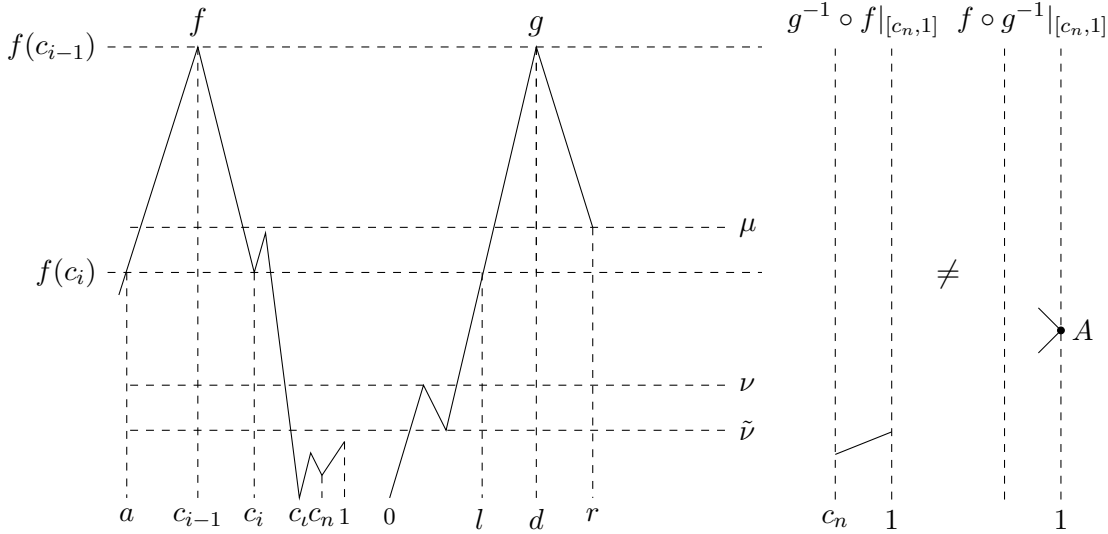

\begin{definition}\label{def:primcritval}
	Let $f\colon I\to I$ be continuous. We say that $v\in I$ is a {\em primary critical value} if there is a critical point $c\in C_f$ such that $f(c)=v$, and if one of the following holds:
	\begin{itemize}
		\item $f^{-1}([0,v))$ and $f^{-1}([v,1])$ are connected, or
		\item $f^{-1}([0,v])$ and $f^{-1}((v,1])$ are connected.
	\end{itemize}
\end{definition}

Note that if $C_f$ is finite, then so is the set of primary critical values, denote them by $0\leq v_1<\ldots<v_k\leq 1$, see Figure~\ref{fig:primcritval}. If $0\not\in f(C_f)$ then denote $v_0=0<v_1$, otherwise $v_1=0$. 

\begin{figure}
	\centering
	\begin{tikzpicture}[scale=5]
	\draw (0,0)--(1,0)--(1,1)--(0,1)--(0,0);
	\draw (0,0)--(1/11,3/11)--(2/11,1/11)--(3/11,5/11)--(4/11,4/11)--(5/11,6/11)--(6/11,5/11)--(7/11,10/11)--(8/11,8/11)--(9/11,9/11)--(10/11,7/11)--(1,1);
	\draw[dashed] (0,1/11)--(1,1/11);
	\draw[dashed] (0,3/11)--(1,3/11);
	\draw[dashed] (0,4/11)--(1,4/11);
	\draw[dashed] (0,6/11)--(1,6/11);
	\draw[dashed] (0,7/11)--(1,7/11);
	\draw[dashed] (0,10/11)--(1,10/11);
	
	\draw[dashed] (1/33,0)--(1/33,1);
	\draw[dashed] (10/44,0)--(10/44,1);
	\draw[dashed] (11/44,0)--(11/44,1);
	\draw[dashed] (31/55,0)--(31/55,1);
	\draw[dashed] (32/55,0)--(32/55,1);
	\draw[dashed] (43/44,0)--(43/44,1);
	
	\node[below] at (0-0.015,0) {\small $t_0$};
	\node[below] at (1/33+0.02,0) {\small $t_1$};
	\node[below] at (10/44-0.015,0) {\small $t_2$};
	\node[below] at (11/44+0.02,0) {\small $t_3$};
	\node[below] at (31/55-0.015,0) {\small $t_4$};
	\node[below] at (32/55+0.025,0) {\small $t_5$};
	\node[below] at (43/44-0.015,0) {\small $t_6$};
	\node[below] at (1+0.02,0) {\small $t_7$};
	
	\node[left] at (0,0) {\small $v_0$};
	\node[left] at (0,1/11) {\small $v_1$};
	\node[left] at (0,3/11) {\small $v_2$};
	\node[left] at (0,4/11) {\small $v_3$};
	\node[left] at (0,6/11) {\small $v_4$};
	\node[left] at (0,7/11) {\small $v_5$};
	\node[left] at (0,10/11) {\small $v_6$};
	\node[left] at (0,1) {\small $v_7$}; 
	\node[below] at (0.5,-0.1) {$(a)$};
	\node[above] at (0.5,1) {$f$};
	\end{tikzpicture}
	\begin{tikzpicture}[scale=5]
	\draw (0,0)--(1,0)--(1,1)--(0,1)--(0,0);
	\draw[dashed] (0,1/11)--(1,1/11);
	\draw[dashed] (0,3/11)--(1,3/11);
	\draw[dashed] (0,4/11)--(1,4/11);
	\draw[dashed] (0,6/11)--(1,6/11);
	\draw[dashed] (0,7/11)--(1,7/11);
	\draw[dashed] (0,10/11)--(1,10/11);
	\draw[dashed] (1/11,0)--(1/11,1);
	\draw[dashed] (3/11,0)--(3/11,1);
	\draw[dashed] (4/11,0)--(4/11,1);
	\draw[dashed] (6/11,0)--(6/11,1);
	\draw[dashed] (7/11,0)--(7/11,1);
	\draw[dashed] (10/11,0)--(10/11,1);
	\node[left] at (0,0) {\small $v_0$};
	\node[left] at (0,1/11) {\small $v_1$};
	\node[left] at (0,3/11) {\small $v_2$};
	\node[left] at (0,4/11) {\small $v_3$};
	\node[left] at (0,6/11) {\small $v_4$};
	\node[left] at (0,7/11) {\small $v_5$};
	\node[left] at (0,10/11) {\small $v_6$};
	\node[left] at (0,1) {\small $v_7$}; 
	\node[below] at (0,0) {\small $v_0$};
	\node[below] at (1/11,0) {\small $v_1$};
	\node[below] at (3/11,0) {\small $v_2$};
	\node[below] at (4/11,0) {\small $v_3$};
	\node[below] at (6/11,0) {\small $v_4$};
	\node[below] at (7/11,0) {\small $v_5$};
	\node[below] at (10/11,0) {\small $v_6$};
	\node[below] at (1,0) {\small $v_7$}; 
	
	\draw (0,0)--(1/11,1/11)--(7/55,3/11)--(9/55,1/11)--(11/55,3/11)--(13/55,1/11)--(3/11,3/11)--(10/33,11/33)--(11/33,10/33)--(4/11,4/11)--(6/11,6/11)--(19/33,20/33)--(20/33,19/33)--(7/11,7/11)--(10/11,10/11)--(21/22,1)--(1,21/22);
	
	\node[below] at (0.5,-0.1) {$(b)$};
	\node[above] at (0.5,1) {$g$};
	\end{tikzpicture}
	\caption{Primary critical values $\{v_i\}_{i=0}^7$ of $f$ and corresponding exacting points $\{t_i\}_{i=0}^7$ in figure $(a)$, and restrictions on map $g$ as in the statement of Lemma~\ref{monotonepatr} and Lemma~\ref{invariant} in $(b)$.}
	\label{fig:primcritval}.
\end{figure}
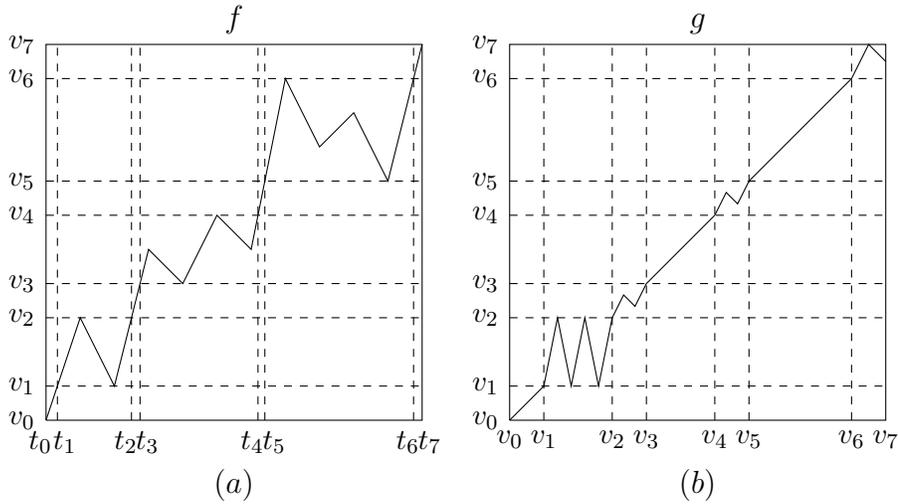

Note that $f^{-1}([v_i,v_{i+1}])=[t_i,t_{i+1}]$ (or $=[t_{i+1},t_i]$) is connected for every odd $i\in\{1,\ldots,k\}$. Also, notice that $f(t_i)=v_i$ and $f(t_{i+1})=v_{i+1}$. Thus it follows that $f$ is either order preserving or order reversing on $\{t_i\}_{i=0\mbox{ or }1}^k$. We call $t_i$ the {\it associated exacting point of} $v_i$ and note that exacting points are not critical points.
Also, since $\{v_i\}_{i=1}^k$ are critical values, there are $c,c'\in C_f\cap f^{-1}([v_i,v_{i+1}])$ such that $f(c)=v_i,f(c')=v_{i+1}$. Furthermore, $f|_{f^{-1}((v_i,v_{i+1}))}$ is monotone for every even $i\in\{0,\ldots,k\}$.

First, we will assume that both $f$ and $g$ are order preserving on their respective exacting points. Later we show that other cases reduce to this case by taking $f^2$ or $g^2$.
Since $f$ is order preserving on the exacting points $\{t_i\}_{i=0\mbox{ or }1}^k$, it holds that $t_i<t_{i+1}$ for every $i$.

\begin{lemma}\label{oddnotcon}
	Let $\{v_i\}_{i=0\mbox{ or }1}^k$ be the primary critical values of $f$. If $i\in\{1,\ldots,k\}$ is odd, and $[p,q]\subsetneq[v_i,v_{i+1}]$ then $f^{-1}([p,q])$ is not connected.
\end{lemma}
\begin{proof} Let $[p,q]\subsetneq[v_i,v_{i+1}]$. Let $c_{a}=\max\{c\in C_f\mid f(c)=v_i\}$ and $c_{b}=\min\{c\in C_f\mid f(c)=v_{i+1}\}$. Let $c_{\alpha}\in [t_i,c_a]\cap C_f$ be such that $f(c_{\alpha})$ is an absolute maximum of $f|_{[t_i,c_a]}$ and $c_{\beta}\in [c_b,t_{i+1}]\cap C_f$ be such that $f(c_{\beta})$ is an absolute minimum of $f|_{[c_b,t_{i+1}]}$. Note that $f(c_{\alpha})\geq f(c_{\beta})$, otherwise $f(c_{\alpha})$ is a primary critical value between $v_i$ and $v_{i+1}$. Suppose $q<f(c_{\alpha})$, then $f^{-1}([p,q])\subset [t_i,c_{\alpha})\cup (c_{\alpha},t_{i+1}]$ with  $f^{-1}([p,q])\cap [t_i,c_{\alpha})\not=\emptyset$ and  $f^{-1}([p,q])\cap (c_{\alpha},t_{i+1})\not=\emptyset$. Hence, $f^{-1}([p,q])$ is not connected. By a similar argument,  $f^{-1}([p,q])$ is not connected if $p>f(c_{\beta})$. So suppose $v_i<p\leq f(c_{\beta})\leq f(c_{\alpha})\leq q$. Then $f^{-1}([p,q])\subset [t_i,c_{a})\cup (c_{a},t_{i+1}]$ with  $f^{-1}([p,q])\cap [t_i,c_{a})\not=\emptyset$ and  $f^{-1}([p,q])\cap (c_{a},t_{i+1})\not=\emptyset$. Hence, $f^{-1}([p,q])$ is not connected. By a similar argument,  $f^{-1}([p,q])$ is not connected if $p\leq f(c_{\beta})\leq f(c_{\alpha})\leq q<v_{i+1}$. 
\end{proof}

\begin{lemma}\label{monotonepatr}
	Let $\{v_i\}_{i=0\mbox{ or }1}^k$ be the primary critical values of $f$. If $i\in\{1,\ldots,k\}$ is odd, then $g^{-1}([v_i,v_{i+1}])$ is connected, and one of the following occurs:
	\begin{enumerate}
		\item $f|_{f^{-1}([v_i,v_{i+1}])}$ is a non-monotone open map, and so is $g|_{g^{-1}([v_i,v_{i+1}])}$, or
		\item $g|_{g^{-1}([v_i,v_{i+1}])}$ is monotone.
	\end{enumerate}
\end{lemma}
\begin{proof}
	First notice that $f|_{f^{-1}([v_i,v_{i+1}])}$ cannot be monotone on $[t_i,t_{i+1}]=f^{-1}([v_i,v_{i+1}])$ for $i$ odd since $[t_i,t_{i+1}]$ must contain critical points. Recall that for odd $i\in\{1,\ldots,k\}$, there are $c,c'\in C_f\cap f^{-1}([v_i,v_{i+1}])$ such that $f(c)=v_i,f(c')=v_{i+1}$. If $g^{-1}([v_i,v_{i+1}])$ is not connected, there are $c_j<c_{j+1}\in C_f$ such that $g^{-1}(f([c_j,c_{j+1}]))$ is not connected, which is in contradiction with Lemma~\ref{lem:connected}. 
	
	Denote $F=f|_{f^{-1}([v_i,v_{i+1}])}\colon f^{-1}([v_i,v_{i+1}])\to[v_i,v_{i+1}], G=g|_{g^{-1}([v_i,v_{i+1}])}\colon g^{-1}([v_i,v_{i+1}])\to[v_i,v_{i+1}]$.
	If $F$ is a non-monotone open map, then Lemma~\ref{lem:open} implies that $G$ is open too. Assume that $F$ is not open, and $G$ is not monotone. Note that Lemma~\ref{lem:open} implies that $G$ also cannot be open. Thus there is $j\in\{0,1,\ldots,m\}$ such that $g([d_j,d_{j+1}])\subsetneq[v_i,v_{i+1}]$. Then it follows from Lemma \ref{oddnotcon} that  that $f^{-1}(g([d_j,d_{j+1}]))$ is not connected, which is  a contradiction.
	\end{proof}

\begin{lemma}\label{prop:prim1} Let $\{v_i\}_{i=0\mbox{ or }1}^k$ be the primary critical points of $f$ and $\{t_i\}_{i=0\mbox{ or }1}^k$ be the associated exacting points. Suppose $i$ is even. If $p\in [t_i,t_{i+1}]$ then $f(x)<f(p)$ for all $x<p$ and $f(p)<f(x)$ for all $x>p$. Furthermore, $f|_{[t_i,t_{i+1}]}$ is monotone increasing.
\end{lemma}
\begin{proof} Follows from the fact that if $c$ is a critical point less than $p$  then $f(c)\leq v_i=f(t_i)\leq f(p)$. Similar, if $c$ is greater than $p$. It now follows that $f|_{[t_i,t_{i+1}]}$ is monotone increasing.
\end{proof}

\begin{lemma}\label{invariant}
	Let $\{v_i\}_{i=0\mbox{ or }1}^k$ be the primary critical values of $f$. Then $g^{-1}([v_i,v_{i+1}])=[v_i,v_{i+1}]$ for every $i\in\{0,\ldots,k\}$. 
	See Figure~\ref{fig:primcritval}.
\end{lemma}
\begin{proof}
	We denote $f^{-1}([v_i,v_{i+1}])=[t_i,t_{i+1}]$, and $g^{-1}([v_i,v_{i+1}])=[s_i,s_{i+1}]$, and note $t_i,t_{i+1}\not\in C_f, s_i,s_{i+1}\not\in C_g$. 
	
	By Lemma~\ref{oddnotcon}, for every odd $i\in\{1,\ldots,k\}$, and every $x\in [v_i,v_{i+1}]\setminus\{0,1\}$, the set $f^{-1}(x)$ consists of at least two points. If $i\in\{0,\ldots,k\}$ is even, then $f|_{[t_i,t_{i+1}]}$ is monotone, so $f^{-1}(x)$ is a single point for every $x\in (t_i,t_{i+1})$. If $0\in f(C_f)$, then $f^{-1}(0)$ can be a single point if and only if $f(c)=0$ for a unique $c\in C_f$, and similarly for $f^{-1}(1)$.
	
	So if $x\in[s_i,s_{i+1}]$ for odd $i\in\{1,\ldots,k\}$, and $g(x)\not\in\{0,1\}$, then $|f^{-1}(g(x))|\geq 2$. If $x\in(s_i,s_{i+1})$ for even $i\in\{0,\ldots,k\}$, then $|f^{-1}(g(x))|=1$.
	
	Note that $(t_i,s_i)\in\Gamma(g^{-1}\circ f)=\Gamma(f\circ g^{-1})$ for every $i\in\{0,\ldots,k\}$, and let $\tau_i$ be such that $(t_i,s_i)=(g(\tau_i),f(\tau_i))$. Since $t_i\not\in C_f$, Lemma~\ref{lem:cross} implies that such $\tau_i$ is unique for every $i\in\{0,\ldots,k\}$. Since $f|_{[t_i,t_{i+1}]}$ is monotone for even $i\in\{0,\ldots,k\}$, it follows that for every point $(x,y)\in\Gamma(g^{-1}\circ f)$ such that $x\in [t_i,t_{i+1}]$ there is a unique $t\in [\tau_i,\tau_{i+1}]$ such that $(x,y)=(g(t),f(t))$. It follows that $f|_{[\tau_i,\tau_{i+1}]}$ is monotone for every even $i\in\{0,\ldots,k\}$.  
	
	Furthermore, let $(x,y)\in\Gamma(g^{-1}\circ f)$ be such that $|f^{-1}(g(y))|\geq 2$. Then for $\tau\in I$ such that $(x,y)=(g(\tau),f(\tau))$ it holds that $|f^{-1}(f(\tau))|\geq 2$. Note that if $(x,y)\in\Gamma(g^{-1}\circ f)$ is such that $|f^{-1}(g(y))|=1$, and $(x,y)=(g(\tau),f(\tau))$ for $\tau\in[\tau_i,\tau_{i+1}]$, where $i$ is odd, then $g(y)\in\{0,1\}$. In particular, the set of all $\tau\in[\tau_i,\tau_{i+1}]$, $i$ odd, for which $f^{-1}(f(\tau))$ is a singleton is finite.
	
	We conclude that $\{\tau: |f^{-1}(f(\tau))|=1\}=\cup_{i \text{ even}}(\tau_i,\tau_{i+1})\cup S$, where $S$ is a finite set. On the other hand, we know that $\{t: |f^{-1}(f(t))|=1\}=\cup_{i \text{ even}}(t_i,t_{i+1})\cup S'$, where $S'\subset\{0,1\}$. Thus it follows that $\tau_i=t_i$ for every $i\in\{0,\ldots,k\}$. Moreover, $f(t_i)=f(\tau_i)=s_i$, and since by definition $f(t_i)=v_i$, it follows that $s_i=v_i$ for every $i\in\{0,\ldots,k\}$, and thus $g^{-1}([v_i,v_{i+1}])=[s_i,s_{i+1}]=[v_i,v_{i+1}]$ for every $i\in\{0,\ldots,k\}$.
\end{proof}

\begin{lemma}\label{prop:prim2} Let $\{v_i\}_{i=0\mbox{ or }1}^k$ be the primary critical points of $f$ and $\{t_i\}_{i=0\mbox{ or }1}^k$ be the associated exacting points.  Suppose $i$ is even. If $p\in [t_i,t_{i+1}]$ and $f(p)=p$ then $f([0,p])=[0,p]$ and $f([p,1])=[p,1]$.
\end{lemma}
\begin{proof} Let $x\in [0,p]$ and $y\in [p,1]$. Then it follows from Lemma \ref{prop:prim1} that $f(x)\leq f(p)=p=f(p)\leq f(y)$. Now the result follows since $f$ is onto.
\end{proof}

\begin{proposition}\label{theor:prim3} Let $f:[a,b]\longrightarrow [c,d]$ be an order preserving homeomorphism where $[c,d]\subset[a,b]$. Then for every $x\in [a,b]$, the limit of $\{f^n(x)\}_{n=0}^{\infty}$ exists and is a fixed point of $f$.
\end{proposition}
\begin{proof} This follows from the Monotone Convergence Theorem.
	\end{proof}

\begin{lemma}\label{lemma:prim4}  Let $x\leq u\leq y\leq v$ and $f:[x,v]\longrightarrow [u,v]$, $g:[u,v]\longrightarrow [u,y]$ be order preserving homeomorphisms such that $f\circ g|_{[u,v]}=g\circ f|_{[u,v]}$. Then there exist $p\in[u,v]$ such that $f(p)=p$ and $g(p)=p$.
\end{lemma}
\begin{proof}
	By Proposition \ref{theor:prim3}, there exists $p=\lim_{n\rightarrow \infty}f^n(u)$ such that $f(p)=p$. Since $g(u)=u$, it follows that $g(p)=\lim_{n\rightarrow \infty}g(f^n(u))=\lim_{n\rightarrow \infty}f^n(g(u))=\lim_{n\rightarrow \infty}f^n(u)=p$.
\end{proof}

\begin{lemma}\label{lemma:prim5}  Let $ u< y\leq v$,  $ u< x<v$ and $f:[u,v]\longrightarrow [x,v]$, $g:[u,v]\longrightarrow [u,y]$ be order preserving homeomorphisms such that $f\circ g=g\circ f$. Then there exist $p\in[u,v]$ such that $f(p)=p$ and $g(p)=p$.
\end{lemma}
\begin{proof}
	Proof is identical to Lemma \ref{lemma:prim4}.
\end{proof}

\begin{lemma}\label{lemma:prim6}
	If $f$ and $g$ are not open, then there exists a fixed point $p$ such that $[a,p]$ and $[p,b]$ are invariant under both $f$ and $g$.
\end{lemma}
\begin{proof}
	Let $\{v_i\}_{i=0,1}^k$ be the primary critical values and $\{t_i\}_{i=0,1}^k$ be the associated exacting points of $f$ and  $\{u_i\}_{i=0,1}^l$ be the primary critical values and $\{s_i\}_{i=0,1}^l$ be the associated exacting points of $g$. 
	
	{\bf Case 1.} There exists  $v_i, u_j\not\in \{0,1\}$ such that  $v_i= u_j$. 
	
	Then it follows from the Lemma~\ref{invariant} that $g(v_i)=v_i=u_j=f(u_j)$. It now follows from Lemma \ref{prop:prim2} that $[0,v_i]$ and $[v_i,1]$ are invariant under both $f$ and $g$.
	
	{\bf Case 2.} If  $v_i, u_j\not\in \{0,1\}$, then  $v_i\not= u_j$. 
	
	Without loss of generality, assume $0\leq u_1<v_1$. Then it follows from Lemma~\ref{monotonepatr} and definition of $v_0$ that $0=v_0\leq u_1<u_2<v_1<u_3$. It also follows from Lemma~\ref{prop:prim1} that whenever $f([a,b])\subset [v_0,v_1]$ and $g([c,d])\subset [u_2,u_3]$, then $f|_{[a,b]}$ and $g|_{[c,d]}$ are both monotone increasing and hence, order preserving homeomorphisms. Note that since $0$ is not a critical value of $f$, $f(0)=0$, and it follows from Lemma~\ref{invariant} that $f(u_1)=u_1$, $f(u_2)=u_2$ and $g(v_1)=v_1$. Furthermore, $f_{[0,t_1]}$ and $g|_{[s_2,s_3]}$ are order preserving homeomorphisms for the same reason.
	
	{\bf Claim:} There exists a $p\in [v_0,v_1]\cap [u_2,u_3]$ such that $f(p)=p=g(p)$.
	
	{\bf Case A.} $s_2\leq u_2$ and $v_1<t_1$.
	
	Then $f(v_1)<f(t_1)=v_1$. So $f|_{[u_2,v_1]}:[u_2,v_1]\longrightarrow [u_2,f(v_1)]$ and $g|_{[s_2,v_1]}:[s_2,v_1]\longrightarrow [u_2,v_1]$ are order preserving homeomorphisms such that $f\circ g|_{[u_2,v_1]}=g\circ f|_{[u_2,v_1]}.$
	
	{\bf Case B.} $s_2\leq u_2$ and $t_1\leq v_1$.
	
	So $f^{-1}|_{[u_2,v_1]}:[u_2,v_1]\longrightarrow [u_2,t_1]$ and $g|_{[s_2,v_1]}:[s_2,v_1]\longrightarrow [u_2,v_1]$ are order preserving homeomorphisms such that $f^{-1}\circ g|_{[u_2,v_1]}=g\circ f^{-1}|_{[u_2,v_1]}.$
	
	{\bf Case C.} $u_2< s_2$ and $v_1<t_1$.
	
	Then $f(v_1)<v_1$ and $g(u_2)\leq g(s_2)=u_2$. So $f|_{[u_2,v_1]}:[u_2,v_1]\longrightarrow [u_2,f(v_1)]$ and $g^{-1}|_{[u_2,v_1]}:[u_2,v_1]\longrightarrow [s_2,v_1]$ are order preserving homeomorphisms such that $f\circ g^{-1}|_{[u_2,v_1]}=g^{-1}\circ f|_{[u_2,v_1]}.$
	
	{\bf Case D.} $u_2< s_2$ and $t_1\leq v_1$.
	
	Then  $g(u_2)<g(s_2)=u_2$. So $f^{-1}|_{[u_2,v_1]}:[u_2,v_1]\longrightarrow [u_2,t_1]$ and $g^{-1}|_{[u_2,v_1]}:[u_2,v_1]\longrightarrow [s_2,v_1]$ are order preserving homeomorphisms such that $f^{-1}\circ g^{-1}|_{[u_2,v_1]}=g^{-1}\circ f^{-1}|_{[u_2,v_1]}$.

	It now follows from Lemmas~\ref{lemma:prim4} and \ref{lemma:prim5} that there exists a $p\in [v_0,v_1]\cap [u_2,u_3]$ such that $f(p)=p=g(p)$. So the result follows again from Lemma~\ref{prop:prim2}.
\end{proof}

\begin{lemma}\label{lem:main_increasing}
	Suppose that $f\colon I\to I$ and $g\colon I\to I$ are piecewise monotone onto maps such that $f^{-1}\circ g=g\circ f^{-1}$ and both $f$ and $g$ are orientation preserving on their exacting points. Then there exists $0=p_0<p_1<\ldots<p_n=1$ such that 
	$[p_i,p_{i+1}]$ is invariant under both $f$ and $g$ and at least one of $\{f|_{[p_i,p_{i+1}]},g|_{[p_i,p_{i+1}]}\}$ is open for each $i$.
\end{lemma}
\begin{proof}
	If $f$ or $g$ is open, then we are done. So suppose that both $f$ and $g$ are not open. Then by Lemma \ref{lemma:prim6}, there exists $0=p_0<p_1<p_2=1$ such that $[p_i,p_{i+1}]$ is invariant under $f$ and $g$ for each $i\in\{0,1\}$. 
	
	Continuing inductively, suppose that $0=p_0<p_1<...<p_k=1$ have been found such that $[p_i,p_{i+1}]$ is invariant under $f$ and $g$ for each $i\in\{0,...,k-1\}$. If at least one of $\{f|_{[p_i,p_{i+1}]},g|_{[p_i,p_{i+1}]}\}$ is open for each $i$, then we are done. Otherwise, there exists $i'$ such that neither $f|_{[p_{i'},p_{i'+1}]}$ nor  $g|_{[p_{i'},p_{i'+1}]}$ is open. Then by Lemma \ref{lemma:prim6} there exists $p_{i'}<p<p_{i'+1}$ such that $[p_{i'},p]$ and $[p,p_{i'+1}]$ are both invariant under $f$ and $g$. Reindex $0=p_0<p_1<\ldots<p_{i'}<p<p_{i'+1}<\ldots<p_k=1$ as $0=p_0<p_1<\ldots<p_k<p_{k+1}=1$.
	
	Eventually for some $n$, this process must stop. Otherwise, $f$ or $g$ will have an infinite number of critical points.
\end{proof}

Now we will consider the case when one of  $f$, $g$ is order preserving and the other is order reversing on their exacting points:

\begin{lemma}\label{lem:subdiv}
	Suppose that $f\colon I\to I$ and $g\colon I\to I$ are  piecewise monotone onto maps such that $f^{-1}\circ g=g\circ f^{-1}$, and such that the following hold:
	\begin{enumerate} 
		\item there exists $0=p_0<p_1<\ldots<p_n=1$ such that 
		$[p_i,p_{i+1}]$ is invariant under both $f^2$ and $g$ and at least one of $\{f^2|_{[p_i,p_{i+1}]},g|_{[p_i,p_{i+1}]}\}$ is open for each $i$,
		\item $\{f(p_i)\}_{i=0}^{\infty}$ is decreasing.
	\end{enumerate}
	Then there exists a subdivision $\{r_i\}_{i=0}^{n'}$, $r_i<r_{i+1}$ such that 
	\begin{enumerate}
		\item $[r_i,r_{i+1}]$ is invariant under both $g$ and $f^2$,
		\item $f([r_i,r_{i+1}])=[r_{n'-i-1},r_{n'-i}] $ for each $i\in\{0,...,n'-1\}$. 
	\end{enumerate}
\end{lemma}
\begin{proof}
	Let $[q_i,q_{i+1}]=f([p_i,p_{i+1}])$. Then notice the following:
	\begin{enumerate}
		\item  $f([q_i,q_{i+1}])=f^2([p_i,p_{i+1}])=[p_i,p_{i+1}]$, thus it follows that $f^2([q_i,q_{i+1}]\cap[p_i,p_{i+1}])=[q_i,q_{i+1}]\cap[p_i,p_{i+1}]$
		\item  $g([q_i,q_{i+1}])=g\circ f([p_i,p_{i+1}])=f\circ g([p_i,p_{i+1}])$ thus it follows that $g([q_i,q_{i+1}]\cap[p_i,p_{i+1}])=[q_i,q_{i+1}]\cap[p_i,p_{i+1}]$.
	\end{enumerate}
	Let  $\{r_i\}_{i=0}^{n'}= \{p_i\}_{i=0}^{n}\cup \{q_i\}_{i=0}^{n}$ such that $0=r_0<r_1<...<r_{n'}=1$. Then it follows from above that $f^2([r_i,r_{i+1}])=[r_i,r_{i+1}]$ and $g([r_i,r_{i+1}])=[r_i,r_{i+1}]$. Next we will  show $f([r_i,r_{i+1}])=[r_{n'-i-1},r_{n'-i}] $. Consider $f([0,r_1])$.
	
	{\bf Case 1.} $r_1=p_1\leq q_{n-1}$. Then $[q_1,1]=f([0,r_1])\subset f([0,q_{n-1}])=[p_{n-1},1]$.
	Hence $r_{n'-1}=q_1$, since $p_{n-1}\leq q_1$.
	
	{\bf Case 2.} $r_1= q_{n-1}\leq p_1$. Then $[p_{n-1},1]=f([0,r_1])\subset f([0,p_1])=[q_1,1]$.
	Hence $r_{n'-1}=p_{n-1}$, since $p_{n-1}\geq  q_1$.	
	
	Now notice that $f([r_1,r_{n'-1}])=[r_1,r_{n'-1}]$. So  it follows by induction  that $f([r_i,r_{i+1}])=[r_{n'-i-1},r_{n'-i}] $ for each $i$.	
\end{proof}

Finally, we will consider the case when both of  $f$ and $g$  is order reversing on their exacting points: 

\begin{lemma}\label{lem:subdiv2}
	Suppose that $f\colon I\to I$ and $g\colon I\to I$ are  piecewise monotone onto maps such that $f^{-1}\circ g=g\circ f^{-1}$, and such that the following hold:
	\begin{enumerate} 
		\item there exists $0=p_0<p_1<\ldots<p_n=1$ such that 
		$[p_i,p_{i+1}]$ is invariant under both $f^2$ and $g^2$ and at least one of $\{f^2|_{[p_i,p_{i+1}]},g^2|_{[p_i,p_{i+1}]}\}$ is open for each $i$,
		\item $\{f(p_i)\}_{i=0}^{\infty}$  and $\{g(p_i)\}_{i=0}^{\infty}$ are decreasing
	\end{enumerate}
	Then there exists a subdivision $\{r_i\}_{i=0}^{n'}$, $r_i<r_{i+1}$ such that 
	\begin{enumerate}
		\item $[r_i,r_{i+1}]$ is invariant under both $g^2$ and $f^2$,
		\item $f([r_i,r_{i+1}])=[r_{n'-i-1},r_{n'-i}] =g([r_i,r_{i+1}])$ for each $i\in\{0,...,n'-1\}$. 
	\end{enumerate}
\end{lemma}
\begin{proof} Let $\alpha=g^2$ and $\beta=f^2$ and notice that $f^{-1}\circ\alpha=\alpha\circ f^{-1}$ and $\beta^{-1}\circ g=g\circ \beta^{-1}$. Then first apply Lemma \ref{lem:subdiv} to $f^{-1}\circ\alpha=\alpha\circ f^{-1}$ to get $\{\hat{r}_i\}_{i=0}^{\hat{n}}$ and then  again apply  Lemma \ref{lem:subdiv} to $\beta^{-1}\circ g=g\circ \beta^{-1}$ get $\{r_i\}_{i=0}^{n'}$.
	\end{proof}

Note that the following theorem is stated for $f^{-1}\circ g=g\circ f^{-1}$, but this is true if and only if $g^{-1}\circ f=f\circ g^{-1}$. So all of the statements are also true when switching $f$ and $g$:
\begin{theorem}\label{thm:main}
	Let $f,g\colon I\to I$ be piecewise monotone onto maps such that $f^{-1}\circ g=g\circ f^{-1}$. Then there are $0=p_0<p_1<\ldots<p_l=1$ such that  one of the following three occurs:
	\begin{enumerate}
		\item[(a)] $[p_i,p_{i+1}]$ is invariant under $f$ and $g$ for every $i$, and
		\begin{itemize}
			\item[(i)] if $g|_{[p_i,p_{i+1}]}$ is open and non-monotone, then $f|_{[p_i,p_{i+1}]}$ is open.
			\item[(ii)] if $g|_{[p_i,p_{i+1}]}$ is non-monotone and not open, then $f|_{[p_i,p_{i+1}]}$ is monotone.
		\end{itemize}
		See Figure~\ref{fig:typea}.
		\item[(b)] $[p_i,p_{i+1}]$ is invariant under  $g$ and $f([p_i,p_{i+1}])=[p_{l-i-1},p_{l-i}]$ for every $i\in\{0,\ldots,l-1\}$, and
		\begin{itemize}
			\item[(i)] if $g|_{[p_i,p_{i+1}]}$ is open and non-monotone, then $f^2|_{[p_i,p_{i+1}]}$ and $f|_{[p_i,p_{i+1}]}$ are open  
			\item[(ii)] if $g|_{[p_i,p_{i+1}]}$ is non-monotone and not open, then $f|_{[p_i,p_{i+1}]}$ and $f|_{[p_{l-i-1},p_{l-i}]}$ are both monotone.
		\end{itemize}
		See Figure~\ref{fig:typeb}.
		\item[(c)]  $f([p_i,p_{i+1}])=[p_{l-i-1},p_{l-i}]=g([p_i,p_{i+1}])$ for every $i\in\{0,\ldots,l-1\}$, and
	\begin{itemize}
		\item[(i)] if $g^2|_{[p_i,p_{i+1}]}$ is open and non-monotone, then $f^2|_{[p_i,p_{i+1}]}$ and $f|_{[p_i,p_{i+1}]}$ are open  
		\item[(ii)] if $g^2|_{[p_i,p_{i+1}]}$ is non-monotone and not open, then $f|_{[p_i,p_{i+1}]}$ and $f|_{[p_{l-i-1},p_{l-i}]}$ are both monotone.
	\end{itemize}
	See Figure~\ref{fig:typec}.
	\end{enumerate}
\end{theorem}
\begin{proof}
	Let $\{v_i\}_{i=0,1}^k$ be primary critical values of $f$, and let $\{t_i\}_{i=0,1}^k$ be the associated exacting points. If $t_i<t_{i+1}$ for every $i\in\{0,1\ldots,k-1\}$, then Lemma~\ref{lem:main_increasing} and Lemma~\ref{monotonepatr} give $(a)$. If $t_i>t_{i+1}$ for every $i\in\{0,1,\ldots,k-1\}$, then let $\{\tilde v_i\}_{i=0,1}^{k'}$ be primary critical values, and $\{\tilde t_i\}_{i=0,1}^{k'}$ the associated exacting points of $f^2$. Then the sequence $\{\tilde t_i\}$ is increasing. Since $f^{-2}\circ g=g\circ f^{-2}$, using $(a)$ we find invariant intervals $[\tilde p_i,\tilde p_{i+1}]$ for both $g$ and $f^2$. Now apply Lemmas~\ref{monotonepatr} and ~\ref{lem:subdiv}. Case $(c)$ is similar to case $(b)$. Instead just use Lemma ~\ref{lem:subdiv2}. 
\end{proof}

\begin{figure}[!ht]
	\centering
	\begin{tikzpicture}[scale=4]
	\draw (0,0)--(0,1)--(1,1)--(1,0)--(0,0);
	\draw (0,0)--(0,1/3)--(1/3,1/3)--(1/3,0)--(0,0);
	\draw (1/3,1/3)--(1/3,2/3)--(2/3,2/3)--(2/3,1/3)--(1/3,1/3);
	\draw (2/3,2/3)--(2/3,1)--(1,1)--(1,2/3)--(2/3,2/3);
	\draw (0,1/3)--(1/6,0)--(1/3,1/3)--(4/9,5/9)--(5/9,4/9)--(2/3,2/3)--(1,1);
	\node[above] at (1/2,1) {\small $f$};
	
	\node[below] at (0,0) {\small $p_0$};
	\node[below] at (1/3,0) {\small $p_1$};
	\node[below] at (2/3,0) {\small $p_2$};
	\node[below] at (1,0) {\small $p_3$};
	\node[left] at (0,0) {\small $p_0$};
	\node[left] at (0,1/3) {\small $p_1$};
	\node[left] at (0,2/3) {\small $p_2$};
	\node[left] at (0,1) {\small $p_3$};
	\end{tikzpicture}
	\begin{tikzpicture}[scale=4]
	\draw (0,0)--(0,1)--(1,1)--(1,0)--(0,0);
	\draw (0,0)--(0,1/3)--(1/3,1/3)--(1/3,0)--(0,0);
	\draw (1/3,1/3)--(1/3,2/3)--(2/3,2/3)--(2/3,1/3)--(1/3,1/3);
	\draw (2/3,2/3)--(2/3,1)--(1,1)--(1,2/3)--(2/3,2/3);
	\draw (0,0)--(1/9,1/3)--(2/9,0)--(1/3,1/3)--(2/3,2/3)--(5/6,1)--(11/12,5/6)--(1,1);
	\node[above] at (1/2,1) {\small $g$};
	
	\node[below] at (0,0) {\small $p_0$};
	\node[below] at (1/3,0) {\small $p_1$};
	\node[below] at (2/3,0) {\small $p_2$};
	\node[below] at (1,0) {\small $p_3$};
	\node[left] at (0,0) {\small $p_0$};
	\node[left] at (0,1/3) {\small $p_1$};
	\node[left] at (0,2/3) {\small $p_2$};
	\node[left] at (0,1) {\small $p_3$};
	\end{tikzpicture}
	\caption{Strongly commuting piecewise monotone maps $f,g\colon I\to I$ of type $(a)$.}
	\label{fig:typea}
\end{figure}
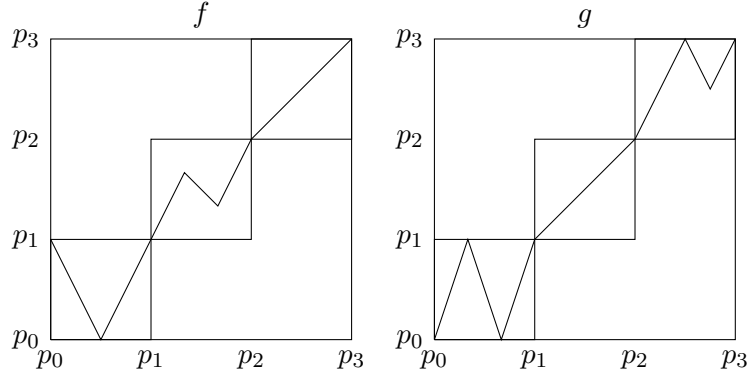
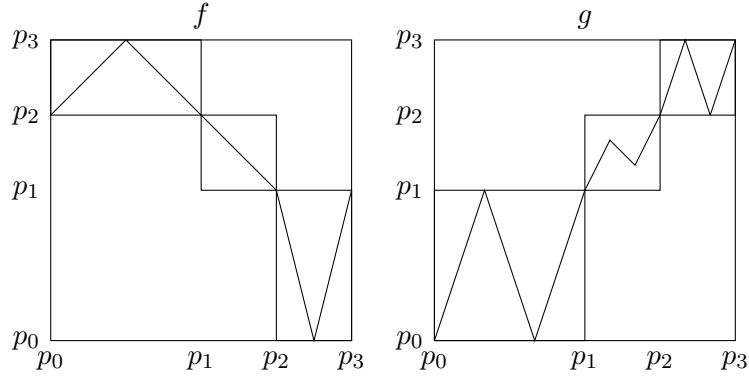
\begin{figure}[!ht]
	\centering
	\begin{tikzpicture}[scale=4]
	\draw (0,0)--(0,1)--(1,1)--(1,0)--(0,0);
	\draw (0,3/4)--(1/2,3/4)--(1/2,1)--(0,1)--(0,3/4);
	\draw (1/2,1/2)--(3/4,1/2)--(3/4,3/4)--(1/2,3/4)--(1/2,1/2);
	\draw (3/4,1/2)--(3/4,0)--(1,0)--(1,1/2)--(3/4,1/2);
	\draw (0,3/4)--(1/4,1)--(1/2,3/4)--(3/4,1/2)--(7/8,0)--(1,1/2);
	\node[above] at (1/2,1) {\small $f$};
	
	\node[below] at (0,0) {\small $p_0$};
	\node[below] at (1/2,0) {\small $p_1$};
	\node[below] at (3/4,0) {\small $p_2$};
	\node[below] at (1,0) {\small $p_3$};
	\node[left] at (0,0) {\small $p_0$};
	\node[left] at (0,1/2) {\small $p_1$};
	\node[left] at (0,3/4) {\small $p_2$};
	\node[left] at (0,1) {\small $p_3$};
	\end{tikzpicture}
	\begin{tikzpicture}[scale=4]
	\draw (0,0)--(0,1)--(1,1)--(1,0)--(0,0);
	\draw (0,0)--(0,1/2)--(1/2,1/2)--(1/2,0)--(0,0);
	\draw (1/2,1/2)--(1/2,3/4)--(3/4,3/4)--(3/4,1/2)--(1/2,1/2);
	\draw (3/4,3/4)--(3/4,1)--(1,1)--(1,3/4)--(3/4,3/4);
	\draw (0,0)--(1/6,1/2)--(2/6,0)--(1/2,1/2)--(7/12,8/12)--(8/12,7/12)--(3/4,3/4)--(10/12,1)--(11/12,3/4)--(1,1);
	\node[above] at (1/2,1) {\small $g$};
	
	\node[below] at (0,0) {\small $p_0$};
	\node[below] at (1/2,0) {\small $p_1$};
	\node[below] at (3/4,0) {\small $p_2$};
	\node[below] at (1,0) {\small $p_3$};
	\node[left] at (0,0) {\small $p_0$};
	\node[left] at (0,1/2) {\small $p_1$};
	\node[left] at (0,3/4) {\small $p_2$};
	\node[left] at (0,1) {\small $p_3$};
	\end{tikzpicture}
	\caption{Strongly commuting piecewise monotone maps $f,g\colon I\to I$ of type $(b)$.}
	\label{fig:typeb}
\end{figure}
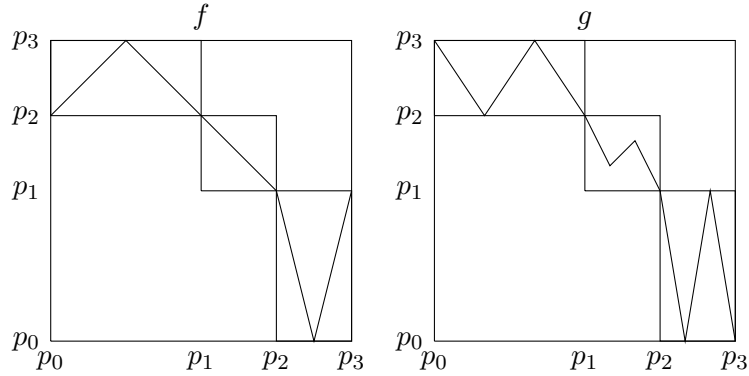
\begin{figure}[!ht]
	\centering
	\begin{tikzpicture}[scale=4]
	\draw (0,0)--(0,1)--(1,1)--(1,0)--(0,0);
	\draw (0,3/4)--(1/2,3/4)--(1/2,1)--(0,1)--(0,3/4);
	\draw (1/2,1/2)--(3/4,1/2)--(3/4,3/4)--(1/2,3/4)--(1/2,1/2);
	\draw (3/4,1/2)--(3/4,0)--(1,0)--(1,1/2)--(3/4,1/2);
	\draw (0,3/4)--(1/4,1)--(1/2,3/4)--(3/4,1/2)--(7/8,0)--(1,1/2);
	\node[above] at (1/2,1) {\small $f$};
	
	\node[below] at (0,0) {\small $p_0$};
	\node[below] at (1/2,0) {\small $p_1$};
	\node[below] at (3/4,0) {\small $p_2$};
	\node[below] at (1,0) {\small $p_3$};
	\node[left] at (0,0) {\small $p_0$};
	\node[left] at (0,1/2) {\small $p_1$};
	\node[left] at (0,3/4) {\small $p_2$};
	\node[left] at (0,1) {\small $p_3$};
	\end{tikzpicture}
	\begin{tikzpicture}[scale=4]
	\draw (0,0)--(0,1)--(1,1)--(1,0)--(0,0);
	\draw (0,3/4)--(1/2,3/4)--(1/2,1)--(0,1)--(0,3/4);
	\draw (1/2,1/2)--(3/4,1/2)--(3/4,3/4)--(1/2,3/4)--(1/2,1/2);
	\draw (3/4,1/2)--(3/4,0)--(1,0)--(1,1/2)--(3/4,1/2);
	\draw (0,1)--(1/6,3/4)--(1/3,1)--(1/2,3/4)--(7/12,7/12)--(8/12,8/12)--(3/4,1/2)--(9/12,1/2)--(10/12,0)--(11/12,1/2)--(1,0);
	\node[above] at (1/2,1) {\small $g$};
	
	\node[below] at (0,0) {\small $p_0$};
	\node[below] at (1/2,0) {\small $p_1$};
	\node[below] at (3/4,0) {\small $p_2$};
	\node[below] at (1,0) {\small $p_3$};
	\node[left] at (0,0) {\small $p_0$};
	\node[left] at (0,1/2) {\small $p_1$};
	\node[left] at (0,3/4) {\small $p_2$};
	\node[left] at (0,1) {\small $p_3$};
	\end{tikzpicture}
	\caption{Strongly commuting piecewise monotone maps $f,g\colon I\to I$ of type $(c)$.}
	\label{fig:typec}
\end{figure}

\begin{remark}
	Let $n\geq 2$ and $f\colon I\to I$ be an open map with $n-1$ critical points. Assume that $f$ is {\em conjugate} to $T_n$, \ie if there is a homeomorphism $h\colon I\to I$ such that $f=h^{-1}\circ T_n\circ h$. Let $g\colon I\to I$ be an onto map which commutes with $f$, and define $\tilde g:=h\circ g\circ h^{-1}$. Since $f$ and $g$ commute, we have $h^{-1}\circ T_n\circ h\circ h^{-1}\circ\tilde g\circ h=h^{-1}\circ\tilde g\circ h\circ h^{-1}\circ T_n\circ h$, so $T_n\circ \tilde g=\tilde g\circ T_n$. By the result of Baxter and Joichi \cite{BaxterJoichi}, the only onto maps which commute with $T_n$ are identity and $T_m$, $m\geq 2$. Thus for every $m\geq 1$ there is a unique onto map $g\colon I\to I$ with $m-1$ critical points which commutes with $f$, and it is exactly $g=h^{-1}\circ T_m\circ h$.
	Also, according to Proposition~\ref{prop:mnrelprime} and Proposition~\ref{prop:mnnotrelprime}, $f$ and $g$ strongly commute if and only if $n$ and $m$ are relatively prime. See for example two maps in Figure~\ref{fig:slantedfolds} which strongly commute and are different (but conjugate with the same conjugacy $h$) to symmetric tent maps $T_3$ and $T_2$.
	
	Maps conjugate to some $T_n$, $n\geq 2$ are called {\em regular} in \cite{Folkman}. For example, if an open map $f$ is {\em locally eventually onto (leo)} (\ie for every open $U\subset I$ there is $n\in\N$ such that $f^n(U)=I$), then $f$ is conjugate to $T_n$, $n\geq 1$, where $|C_f|=n-1$.
\end{remark}

\begin{figure}
	\centering
	\begin{minipage}{0.5\textwidth}
		\begin{tikzpicture}[scale=6]
		\draw (0,0)--(0,1)--(1,1)--(1,0)--(0,0);
		\draw[very thick] (0,0)--(1/4,1)--(1/2,0)--(1,1);
		\draw (1/4,-0.01)--(1/4,0.01);
		\node[below] at (1/4,0) {\tiny $0.25$};
		\draw (1/2,-0.01)--(1/2,0.01);
		\node[below] at (1/2,0) {\tiny $0.5$};
		\node at (0.5,1.09) {};
		\end{tikzpicture}
	\end{minipage}
	\hspace{-58pt}
	\begin{minipage}{0.5\textwidth}
			\scalebox{0.65}{\input{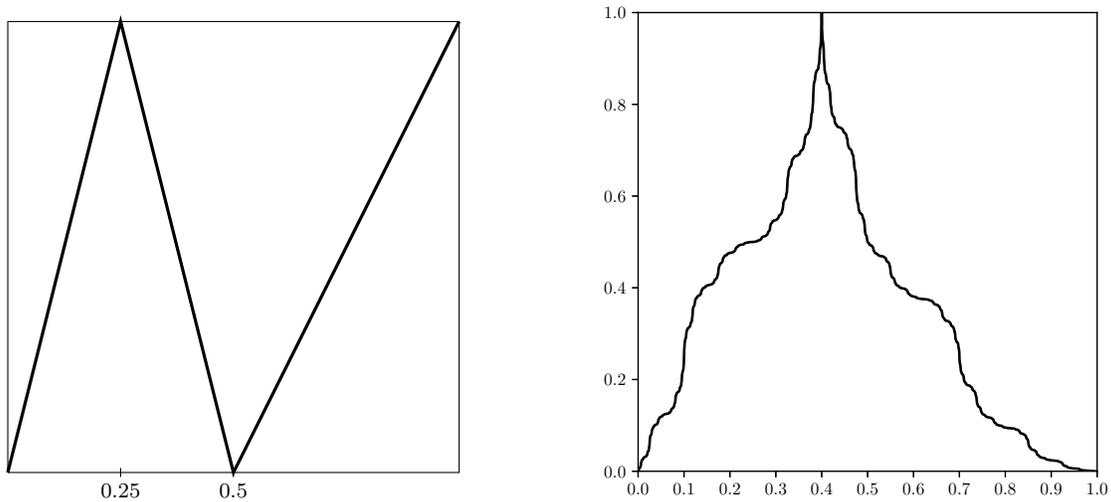}}
	\end{minipage}
	\vspace{-20pt}
	\caption{Map $f$ on the left is leo, so there is a homeomorphism $h\colon I\to I$ such that $f=h^{-1}\circ T_3\circ h$. On the right we draw the graph of $g=h^{-1}\circ T_2\circ h$, which is the unique map with a single critical point which (strongly) commutes with $f$.}
	\label{fig:slantedfolds}
\end{figure}

\begin{corollary}
	Let $f,g\colon I\to I$ be piecewise monotone, strongly commuting maps. Then there is $x\in I$ such that $f(x)=g(x)=x$.
\end{corollary}
\begin{proof}
	If one of $f$, $g$ is monotone, then the result of Joichi \cite{Joichi} gives the common fixed point $x$. Otherwise $l>1$ in Theorem~\ref{thm:main}. Case $(a)$ implies that $f(p_1)=g(p_1)=p_1$ (or we get a common fixed point directly from Lemma~\ref{lemma:prim6}). Similarly, if $l$ is even in case $(b)$, we again get $f(p_{l/2})=g(p_{l/2})=p_{l/2}$. If $l$ is odd, then $[p_{\frac{l-1}{2}},p_{\frac{l+1}{2}}]$ is invariant for both $f$ and $g$. We apply Joichi's theorem to find a common fixed point for $f$ and $g$ in $[p_{\frac{l-1}{2}},p_{\frac{l+1}{2}}]$. Case $(c)$ is similar to case $(b)$.
\end{proof}

\end{document}